\documentclass[12pt]{amsart}
\usepackage{lmodern}
\usepackage{ifthen}
\usepackage{amsfonts}
\usepackage{amsxtra}
\usepackage{amssymb}
\usepackage{array}
\usepackage[margin=1.0in]{geometry}
\usepackage{xcolor}
\definecolor{cite}{rgb}{0.30,0.60,1.00}
\definecolor{url}{rgb}{0.00,0.00,0.80}
\definecolor{link}{rgb}{0.90,0.10,0.20}
\usepackage[colorlinks,linkcolor=link,urlcolor=url,citecolor=cite,pagebackref,breaklinks]{hyperref}
\usepackage{bbm}
\usepackage{mathtools}
\usepackage{mathrsfs}
\usepackage{appendix}
\usepackage[all]{xy}
\usepackage[lite,abbrev,msc-links,alphabetic]{amsrefs}
\usepackage{enumerate}

\usepackage[OT2,T1]{fontenc}
\DeclareSymbolFont{cyrletters}{OT2}{wncyr}{m}{n}
\DeclareMathSymbol{\Sha}{\mathalpha}{cyrletters}{"58}

\usepackage[utf8]{inputenc}

\numberwithin{equation}{section}

\theoremstyle{plain}
\newtheorem{proposition}{Proposition}[section]
\newtheorem{conjecture}[proposition]{Conjecture}
\newtheorem{corollary}[proposition]{Corollary}
\newtheorem{lem}[proposition]{Lemma}
\newtheorem{theorem}[proposition]{Theorem}

\theoremstyle{definition}
\newtheorem{definition}[proposition]{Definition}
\newtheorem{construction}[proposition]{Construction}
\newtheorem{notation}[proposition]{Notation}

\theoremstyle{remark}
\newtheorem{remark}[proposition]{Remark}

\renewcommand{\b}[1]{\mathbf{#1}}
\renewcommand{\c}[1]{\mathcal{#1}}

\newcommand{\f}[1]{\mathfrak{#1}}
\renewcommand{\r}[1]{\mathrm{#1}}
\newcommand{\s}[1]{\mathscr{#1}}

\renewcommand{\(}{\left(}
\renewcommand{\)}{\right)}
\newcommand{\res}{\mathbin{|}}

\newcommand{\Sec}{\S}

\newcommand{\bC}{\b C}

\newcommand{\bG}{\b G}

\newcommand{\bP}{\b P}
\newcommand{\bQ}{\b Q}
\newcommand{\bR}{\b R}

\newcommand{\bZ}{\b Z}

\newcommand{\cC}{\c C}
\newcommand{\cD}{\c D}

\newcommand{\cI}{\c I}

\newcommand{\cV}{\c V}

\newcommand{\cX}{\c X}

\newcommand{\fX}{\f X}

\newcommand{\rH}{\r H}
\newcommand{\rI}{\r I}
\newcommand{\rJ}{\r J}

\newcommand{\rN}{\r N}

\newcommand{\rT}{\r T}

\newcommand{\ra}{\r a}

\newcommand{\rd}{\r d}

\newcommand{\rs}{\r s}

\newcommand{\sA}{\s A}

\newcommand{\sC}{\s C}

\newcommand{\sF}{\s F}

\newcommand{\sL}{\s L}

\newcommand{\sN}{\s N}
\newcommand{\sO}{\s O}

\newcommand{\sT}{\s T}

\newcommand{\an}{\r{an}}

\newcommand{\cl}{\r{cl}}

\newcommand{\cons}{\r{cons}}

\newcommand{\et}{\acute{\r{e}}\r{t}}

\newcommand{\trop}{\r{trop}}

\newcommand{\ur}{\r{ur}}

\DeclareMathOperator{\Char}{char}

\DeclareMathOperator{\DIV}{div}

\DeclareMathOperator{\Gal}{Gal}

\DeclareMathOperator{\IM}{im}

\DeclareMathOperator{\Ker}{ker}

\DeclareMathOperator{\Map}{Map}

\DeclareMathOperator{\ord}{ord}

\DeclareMathOperator{\Pic}{Pic}

\DeclareMathOperator{\Spec}{Spec}

\DeclareMathOperator{\Tr}{Tr}

\begin{document}

\title{Monodromy map for tropical Dolbeault cohomology}

\author{Yifeng Liu}
\address{Department of Mathematics, Northwestern University, Evanston IL 60208, United States}
\email{liuyf@math.northwestern.edu}

\date{\today}
\subjclass[2010]{14G22}

\begin{abstract}
  We define monodromy maps for tropical Dolbeault cohomology of algebraic varieties over non-Archimedean fields. We propose a conjecture of Hodge isomorphisms via monodromy maps, and provide some evidence.
\end{abstract}

\maketitle

\setcounter{tocdepth}{1}
\tableofcontents

\section{Introduction}
\label{ss:1}

Let $K$ be a complete non-Archimedean field. For an algebraic variety $X$ over $K$, one has the associated $K$-analytic space $X^\an$ in the sense of Berkovich \cite{Berk93}. Using a bicomplex $\sA^{\bullet,\bullet}_{X^\an}$ of $\bR$-sheaves of real forms on $X^\an$ concentrated in the first quadrant, constructed by Chambert-Loir and Ducros \cite{CLD12}, one can define, analogous to complex manifolds, Dolbeault cohomology groups $H^{p,q}_\trop(X)$. We call them \emph{tropical Dolbeault cohomology}. They are real vector spaces satisfying $H^{p,q}_\trop(X)\neq 0$ only when $0\leq p,q\leq\dim X$. In this article, we consider a natural question, motivated from complex geometry, about the Hodge isomorphism: Do we have $H^{p,q}_\trop(X)\simeq H^{q,p}_\trop(X)$?

However, the perspective we take is very different from complex geometry. In fact, we will construct for $p\geq 1$ and $q\geq 0$ a functorial map
\[\rN_X\colon H^{p,q}_\trop(X)\to H^{p-1,q+1}_\trop(X)\]
which we call the \emph{monodromy map}, via constructing a canonical map $\sA^{p,q}_{X^\an}\to\sA^{p-1,q+1}_{X^\an}$ on the level of sheaves. Such map does not exist in complex geometry. We conjecture (see Conjecture \ref{co:monodromy}) that over certain field $K$, the iterated map
\[\rN_X^{p-q}\colon H^{p,q}_\trop(X)\to H^{q,p}_\trop(X)\]
is an isomorphism for $p\geq q$ (assuming $X$ is complete and smooth). We view such isomorphism as the Hodge isomorphism for tropical Dolbeault cohomology. If this holds, then we have the following numerical relation for corresponding Hodge numbers:
\[h^{p,0}_\trop(X)\leq h^{p-1,1}_\trop(X)\leq \cdots \cdots \geq h^{1,p-1}_\trop(X)\geq h^{0,p}_\trop(X).\]
This is apparently new from complex geometry. In fact, we should compare the Hodge isomorphism and the above numerical relation with the monodromy-weight conjecture for the $E_2$ page of the weight spectral sequence (Conjecture \ref{co:purity}). We remark that the map $\rN_X^p\colon H^{p,0}_\trop(X)\to H^{0,p}_\trop(X)$ is simply induced by the ``flipping'' map $\rJ$ \cite{CLD12} multiplied by $p!$ (Lemma \ref{le:functorial} (1)).

The theorem below provides certain evidence toward our expected Hodge isomorphism. The proof uses arithmetic geometry, especially the weight spectral sequence. The method is restrictive in the sense that any further extension would seem to be relied on the monodromy-weight conjecture in the mixed characteristic case, among other difficulties. Therefore, it is very interesting and important to find a somewhat analytic proof of the results here. However, the situation is rather subtle since we know for some field $K$, such as the completed algebraic closure of $\bC(\!(t)\!)$, the map $\rN_X\colon H^{1,0}_\trop(X)\to H^{0,1}_\trop(X)$ may fail to be an isomorphism (see Remark \ref{re:counter}).\\

In this article, by a non-Archimedean field, we mean a complete topological field equipped with a nontrivial rank-$1$ non-Archimedean valuation. For a non-Archimedean field $K$, we denote by $K^\circ$ the ring of integers, $\widetilde{K}$ the residue field, $K^\ra$ a fixed algebraic closure with $\widehat{K^\ra}$ its completion.

\begin{theorem}
Let $X_0$ be a proper smooth scheme over a non-Archimedean field $K_0$. Let $K$ be a closed subfield of $\widehat{K_0^\ra}$ containing $K_0$. Put $X=X_0\otimes_{K_0}K$.
\begin{enumerate}
  \item (Corollary \ref{co:injective}) Suppose that $K_0$ is isomorphic to $k(\!(t)\!)$ for $k$ either a finite field or a field of characteristic zero. Then the monodromy map
      \[\rN_X^p\colon H^{p,0}_\trop(X)\to H^{0,p}_\trop(X)\]
      is injective for every $p\geq 0$. In particular, $H^{p,0}_\trop(X)$ is of finite dimension.

  \item (Theorem \ref{th:isomorphic}) Suppose that $K_0$ is either a finite extension of $\bQ_p$ or isomorphic to $k(\!(t)\!)$ for $k$ a finite field, $K=\widehat{K_0^\ra}$, and $X_0$ admits a proper strictly semistable model over $K_0^\circ$. Then the monodromy map
      \[\rN_X\colon H^{1,0}_\trop(X)\to H^{0,1}_\trop(X)\]
      is an isomorphism.
\end{enumerate}
\end{theorem}

\begin{remark}\label{re:jell}
Let $K$ be an algebraically closed non-Archimedean field.
\begin{enumerate}
  \item In his thesis, Jell proved that for a proper smooth scheme $X$ over $K$ of dimension $n$, the map $\rN_X^p\colon H^{p,0}_\trop(X)\to H^{0,p}_\trop(X)$ is injective for $p=0,1,n$ \cite{Jel}*{Proposition 3.4.11}.

  \item In \cite{JW16}, Jell and Wanner proved that for $X$ either $\bP^1_K$ or a (proper smooth) Mumford curve over $K$, the map $\rN_X\colon H^{1,0}_\trop(X)\to H^{0,1}_\trop(X)$ is an isomorphism.
\end{enumerate}
Their methods in both papers are analytic.
\end{remark}

In \cite{MZ13}, Mikhalkin and Zharkov study the tropical homology groups $H_{p,q}(X)$ of a \emph{tropical space} $X$. Assuming $X$ compact, they define a map
\[\phi\cap\colon H_{p,q}(X)\otimes\bR\to H_{p+1,q-1}(X)\otimes\bR\]
using their eigenwave $\phi$. In fact, as shown in \cite{JSS15}, one can compute the tropical homology $H_{p,q}(X)\otimes\bR$, or rather the tropical cohomology, via superforms in \cite{CLD12} as well. Then our construction in \Sec\ref{ss:2} would give rise to a map $\rN_X\colon H_{p,q}(X)\otimes\bR\to H_{p+1,q-1}(X)\otimes\bR$ (or rather on cohomology) for any tropical space $X$. We expect that $\rN_X$ and $\phi\cap$ should coincide when $X$ is compact, possibly up to an elementary factor.

Moreover, in \cite{MZ13}, the authors prove that when $X$ is a realizable smooth compact tropical space, the iterated map $\phi^{q-p}\cap\colon H_{p,q}(X)\otimes\bR\to H_{q,p}(X)\otimes\bR$ is an isomorphism. They first realize $X$ as the tropical limit of a complex projective one-parameter semistable degeneration $\cX$ such that all strata of the singular fiber $\cX_0$ are blow-ups of projective spaces. By the work \cite{IKMZ16}, one knows that $H_{p,q}(X)\otimes\bQ$ can be identified with $E_2$-terms of the Steenbrink--Illusie spectral sequence associated to $\cX$. Moreover, Mikhalkin and Zharkov show that under such identification, the map $\phi\cap$ is simply the monodromy map on $E_2$-terms. In our work, $X$ is a proper smooth\footnote{It is worth noting that the tropicalization of a smooth analytic space, such as $X^\an$, under a local tropical chart is in general not a smooth tropical space.} \emph{algebraic variety} and we relate our tropical Dolbeault cohomology $H^{p,q}_\trop(X)$ to $E_2$-terms of the weight spectral sequence of (various) semistable alteration $\cX$ of $X$ (see Definition \ref{de:semistable_alteration} for the precise meaning) for certain $p,q$, under which $\rN_X$ is essentially the monodromy map. In our setup, the semistable scheme $\cX$ is very general, so that its $E_2$-terms cannot be read off solely from the dual complex of $\cX$, hence we do not obtain a strict identification for general $p,q$. Nevertheless, it would be interesting to compare our approach relating to the weight spectral sequence and their approach relating to the Steenbrink--Illusie spectral sequence in \cite{MZ13,IKMZ16}.

\subsection*{Convention}

For an analytic space $X$ over $K$, we use $H^\bullet(X,-)$ to indicate the cohomology group with respect to the underlying topology of $X$, and $H^\bullet_{\et}(X,-)$ to indicate the cohomology group with respect to the \'{e}tale topology of $X$.

\subsection*{Acknowledgements}

The author would like to thank Walter~Gubler, Philipp~Jell, and Klaus~K\"{u}nnemann for helpful discussions and their hospitality during his visit at Universit\"{a}t Regensburg. He would also like to thank Erwan~Brugall\'{e} and Philipp~Jell for drawing his attention to the work of Mikhalkin and Zharkov \cite{MZ13}.

\section{Monodromy maps for superforms on vector spaces}
\label{ss:2}

In this section, we review the construction of superforms on vector spaces from \cite{CLD12} and introduce the corresponding monodromy map.

Let $V$ be an $\bR$-vector space of dimension $n$. Let $T_V$ be the tangent space of $V$ and $T_V^*$ its dual space. For every open subset $U$ of $V$ and integers $p,q\geq 0$, we have the space of $(p,q)$-forms on $U$,
\[
\sA_V^{p,q}(U)\coloneqq\sA_V(U)\otimes_\bR\wedge^pT_V^*\otimes_\bR\wedge^qT_V^*,
\]
where $\sA_V(U)$ is the $\bR$-algebra of smooth functions on $U$. The direct sum $\sA_V^{*,*}(U)\coloneqq\bigoplus\sA_V^{p,q}(U)$ form a bigraded $\bR$-algebra with the following commutativity law: if $\omega$ is a $(p,q)$-form and $\omega'$ is a $(p',q')$-form, then
\[\omega'\wedge\omega=(-1)^{(p+q)(p'+q')}\omega\wedge\omega'.\]
Moreover, $\wedge\colon\sA_V^{p,0}(U)\times\sA_V^{0,q}(U)\to\sA_V^{p,q}(U)$ is simply the one induced by the tensor product map.

The real vector spaces $\sA_V^{\bullet,\bullet}(U)$ form a bicomplex with two differentials
\[
\rd'\colon \sA_V^{p,q}(U)\to \sA_V^{p+1,q}(U),\qquad
\rd''\colon \sA_V^{p,q}(U)\to \sA_V^{p,q+1}(U).
\]
We also have a convolution map $\rJ\colon \sA_V^{p,q}(U)\to \sA_V^{q,p}(U)$.

In terms of coordinates, they are described as follows. Let $\{x_1,\dots,x_n\}$ be a system of coordinates of $V$. Then a $(p,q)$-form can be written as
\[\omega=\sum_{|I|=p,|J|=q}\omega_{I,J}(x)\rd'x_I\otimes\rd''x_J=\sum_{|I|=p,|J|=q}\omega_{I,J}(x)\rd'x_I\wedge\rd''x_J\]
where $I,J$ are subsets of $\{1,\dots,n\}$, and $\omega_{I,J}(x)$ are smooth functions on $U$. For such $\omega$, we have
\begin{align*}
\rd'\omega &=\sum_{|I|=p,|J|=q}\sum_{i=1}^n\frac{\partial\omega_{I,J}(x)}{\partial x_i}\rd'x_i\wedge\rd'x_I\wedge\rd''x_J, \\
\rd''\omega &=(-1)^p\sum_{|I|=p,|J|=q}\sum_{j=1}^n\frac{\partial\omega_{I,J}(x)}{\partial x_j}\rd'x_I\wedge\rd''x_j\wedge\rd''x_J, \\
\rJ\omega &= (-1)^{pq}\sum_{|I|=p,|J|=q}\omega_{I,J}(x)\rd'x_J\wedge\rd''x_I.
\end{align*}
We remind readers from \cite{CLD12}*{\Sec 1.2} the relations
\[\rd''=\rJ\rd'\rJ,\quad \rJ\rd'=\rd''\rJ,\quad \rJ\rd''=\rd'\rJ,\quad \rd'=\rJ\rd''\rJ,\]
and for a $(p,q)$-form $\omega$ and a $(p',q')$-form $\omega'$,
\begin{align*}
\rd'(\omega\wedge\omega')&=\rd'\omega\wedge\omega'+(-1)^{p+q}\omega\wedge\rd'\omega',\\
\rd''(\omega\wedge\omega')&=\rd''\omega\wedge\omega'+(-1)^{p+q}\omega\wedge\rd''\omega'.
\end{align*}

Now we are going to define a map $\rN\colon\sA_V^{p,q}(U)\to\sA_V^{p-1,q+1}(U)$ for $p\geq 1$. We first recall the notion of coevaluation map. Let $W$ be an arbitrary finite dimensional $\bR$-vector space with $W^*$ its dual space. We have a canonical evaluation map
\[\r{ev}\colon W^*\otimes_\bR W\to\bR.\]
We define the \emph{coevaluation map} to be the unique linear map
\[\r{coev}\colon \bR\to W\otimes_\bR W^*\]
such that both composite maps
\begin{align*}
W^*\xrightarrow{1_{W^*}\otimes\r{coev}}W^*\otimes_\bR(W\otimes_\bR W^*)
&\xrightarrow{\sim}(W^*\otimes_\bR W)\otimes_\bR W^*\xrightarrow{\r{ev}\otimes1_{W*}}W^*\\
W\xrightarrow{\r{coev}\otimes1_W}(W\otimes_\bR W^*)\otimes_\bR W
&\xrightarrow{\sim}W\otimes_\bR(W^*\otimes_\bR W)\xrightarrow{1_W\otimes\r{ev}}W
\end{align*}
are identity maps. Now we apply the coevaluation map to the vector space $W=T_V$.

\begin{definition}\label{de:monodromy_local}
Let $p\geq 1$ be an integer. Define the map
\[\rN\colon\sA_V^{p,q}(U)\to \sA_V^{p-1,q+1}(U)\]
to be the composite map
\begin{align*}
\sC^\infty(U)\otimes_\bR\wedge^pT_V^*\otimes_\bR\wedge^qT_V^*
&\xrightarrow{\sim}\sC^\infty(U)\otimes_\bR\wedge^pT_V^*\otimes_\bR\bR\otimes_\bR\wedge^qT_V^*\\
&\to\sC^\infty(U)\otimes_\bR\wedge^pT_V^*\otimes_\bR(T_V\otimes_\bR T_V^*)\otimes_\bR\wedge^qT_V^*\\
&\xrightarrow{\sim}\sC^\infty(U)\otimes_\bR(\wedge^pT_V^*\otimes_\bR T_V)\otimes_\bR(T_V^*\otimes_\bR\wedge^qT_V^*)\\
&\to\sC^\infty(U)\otimes_\bR\wedge^{p-1}T_V^*\otimes_\bR\wedge^{q+1}T_V^*,
\end{align*}
where the second map is given by the coevaluation map for $T_V$, and the last map is given by the contraction map and the wedge product.

For $0\leq r\leq p$, we denote by $\rN^r\colon\sA_V^{p,q}(U)\to\sA_V^{p-r,q+r}(U)$ the consecutive composition.
\end{definition}

In terms of the coordinates $\{x_1,\dots,x_n\}$ of $V$, we have for
\begin{align}\label{eq:omega}
\omega=\sum_{I=\{i_1<\cdots<i_p\},J=\{j_1<\cdots<j_q\}}
\omega_{I,J}(x)\rd'x_{i_1}\wedge\cdots\wedge\rd'x_{i_p}\wedge\rd''x_{j_1}\wedge\cdots\wedge\rd''x_{j_q}
\end{align}
with $p\geq 1$ that
\begin{align*}
\rN\omega
&=\sum_{k=1}^p\sum_{I,J}
(-1)^{p-k}\omega_{I,J}(x)\rd'x_{i_1}\wedge\cdots\wedge\widehat{\rd'x_{i_k}}\wedge\cdots\wedge\rd'x_{i_p}
\wedge\rd''x_{i_k}\wedge\rd''x_{j_1}\wedge\cdots\wedge\rd''x_{j_q}\\
&=\sum_{k=1}^p\sum_{I,J}
(-1)^{p-k}\omega_{I,J}(x)\rd'x_{I\setminus\{i_k\}}\wedge\rd''x_{i_k}\wedge\rd''x_J.
\end{align*}

\begin{lem}\label{le:differential_local}
We have $\rN\rd''=\rd''\rN\colon\sA_V^{p,q}(U)\to\sA_V^{p-1,q+2}(U)$ for $p\geq 1$.
\end{lem}

\begin{proof}
Take a $(p,q)$-form $\omega$ as \eqref{eq:omega}. We have
\begin{align*}
\rd''\rN\omega&=\rd''\sum_{k=1}^p\sum_{I,J}(-1)^{p-k}
\omega_{I,J}(x)\rd'x_{I\setminus\{i_k\}}\wedge\rd''x_{i_k}\wedge\rd''x_J\\
&=-\sum_{I,J}\sum_{j=1}^n\sum_{k=1}^p\frac{(-1)^k\partial\omega_{I,J}(x)}{\partial x_j}
\rd'x_{I\setminus\{i_k\}}\wedge\rd''x_j\wedge\rd''x_{i_k}\wedge\rd''x_J.
\end{align*}
On the other hand, we have
\begin{align*}
\rN\rd''\omega&=\rN(-1)^p\sum_{I,J}\sum_{j=1}^n\frac{\partial\omega_{I,J}(x)}{\partial x_j}
\rd'x_I\wedge\rd''x_j\wedge\rd''x_J\\
&=\sum_{k=1}^p\sum_{I,J}\sum_{j=1}^n\frac{(-1)^k\partial\omega_{I,J}(x)}{\partial x_j}
\rd'x_{I\setminus\{i_k\}}\wedge\rd''x_{i_k}\wedge\rd''x_j\wedge\rd''x_J\\
&=-\sum_{I,J}\sum_{j=1}^n\sum_{k=1}^p\frac{(-1)^k\partial\omega_{I,J}(x)}{\partial x_j}
\rd'x_{I\setminus\{i_k\}}\wedge\rd''x_j\wedge\rd''x_{i_k}\wedge\rd''x_J.
\end{align*}
The lemma follows.
\end{proof}

\begin{lem}\label{le:wedge_local}
For a $(p,q)$-form $\omega$ and a $(p',q')$-form $\omega'$ with $p,p'\geq 1$ and $p+p'\geq n+1$, we have
\[\rN\omega\wedge\omega'=-\omega\wedge\rN\omega'.\]
\end{lem}

\begin{proof}
By linearity, we may assume that
\[\omega=\omega(x)\rd'x_I\wedge\rd''x_J,\qquad\omega'=\omega'(x)\rd'x_{I'}\wedge\rd''x_{J'}.\]
If $|I\cap I'|\geq 2$, then it is easy to see that $\rN\omega\wedge\omega'=\omega\wedge\rN\omega'=0$. Otherwise, $|I\cap I'|=1$. Without lost of generality, we may assume that $I\cap I'=\{n\}$ hence $\rd'x_I=\rd'x_{I\setminus\{n\}}\wedge\rd'x_n$ and $\rd'x_{I'}=\rd'x_{I'\setminus\{n\}}\wedge\rd'x_n$. Then we have
\begin{align*}
\rN\omega\wedge\omega'&=(\omega(x)\rd'x_{I\setminus\{n\}}\wedge\rd''x_n\wedge\rd''x_J)\wedge(\omega'(x)\rd'x_{I'}\wedge\rd''x_{J'})\\
&=(-1)^{p'(q+1)}\omega(x)\omega'(x)\rd'x_{I\setminus\{n\}}\wedge\rd'x_{I'}\wedge\rd''x_n\wedge\rd''x_J\wedge\rd''x_{J'}\\
&=(-1)^{p'(q+1)}
\omega(x)\omega'(x)\rd'x_{I\setminus\{n\}}\wedge\rd'x_{I'\setminus\{n\}}\wedge\rd'x_n\wedge\rd''x_n\wedge\rd''x_J\wedge\rd''x_{J'};
\end{align*}
and
\begin{align*}
\omega\wedge\rN\omega'&=(\omega(x)\rd'x_I\wedge\rd''x_J)\wedge(\omega'(x)\rd'x_{I'\setminus\{n\}}\wedge\rd''x_n\wedge\rd''x_{J'})\\
&=(-1)^{p'q}\omega(x)\omega'(x)\rd'x_I\wedge\rd'x_{I'\setminus\{n\}}\wedge\rd''x_n\wedge\rd''x_J\wedge\rd''x_{J'}\\
&=(-1)^{p'q}\omega(x)\omega'(x)\rd'x_{I\setminus\{n\}}\wedge\rd'x_n\wedge\rd'x_{I'\setminus\{n\}}\wedge\rd''x_n\wedge\rd''x_J\wedge\rd''x_{J'}\\
&=(-1)^{p'q+(p'-1)}
\omega(x)\omega'(x)\rd'x_{I\setminus\{n\}}\wedge\rd'x_{I'\setminus\{n\}}\wedge\rd'x_n\wedge\rd''x_n\wedge\rd''x_J\wedge\rd''x_{J'}.
\end{align*}
The lemma follows.
\end{proof}

\begin{lem}\label{le:functorial_local}
Let $V'$ be another $\bR$-vector space of dimension $n'$ and $U'\subset V'$ an subset. Let $\varphi\colon V'\to V$ be an affine map such that $\varphi(U')\subset U$. Then for $p\geq 1$, we have
\[\rN'\varphi^*=\varphi^*\rN\colon\sA_V^{p,q}(U)\to\sA_{V'}^{p-1,q+1}(U')\]
where $\rN'$ denotes the monodromy map for $V'$.
\end{lem}

\begin{proof}
We may assume that $\varphi$ is a linear map. It suffices to consider cases where $\varphi$ is injective or surjective.

Suppose that $\varphi$ is injective. Regard $V'$ as a subspace of $V$. Choose a basis $\{x_1,\dots,x_n\}$ of $V$ such that $V'$ is spanned by $\{x_1,\dots,x_{n'}\}$ (with $n'\leq n$). Take $\omega=\omega(x)\rd'x_I\wedge\rd''x_J\in\sA_V^{p,q}(U)$ for $I=\{i_1<\cdots<i_p\}$. If $i_p>n'$, then $\rN'\varphi^*\omega=\varphi^*\rN\omega=0$. If $i_p\leq n'$, then we have
\begin{align*}
\rN'\varphi^*\omega=\varphi^*\rN\omega=\(\omega(x)\sum_{k=1}^p(-1)^{p-k}\rd'x_{I\setminus\{i_k\}}\wedge\rd''x_{i_k}\wedge\rd''x_J\)\res_{U'}.
\end{align*}

Suppose that $\varphi$ is surjective. Choose a basis $\{x_1,\dots,x_{n'}\}$ of $V'$ such that $\Ker\varphi$ is spanned by $\{x_{n+1},\dots,x_{n'}\}$ (with $n\leq n'$). We identify $V$ with the subspace of $V'$ spanned by $\{x_1,\dots,x_n\}$. Then again we have
\begin{align*}
\rN'\varphi^*\omega=\varphi^*\rN\omega=
\(\omega(\varphi(x))\sum_{k=1}^p(-1)^{p-k}\rd'x_{I\setminus\{i_k\}}\wedge\rd''x_{i_k}\wedge\rd''x_J\)\res_{U'}.
\end{align*}
The lemma follows.
\end{proof}

\begin{lem}\label{le:J_pre}
The map $\rN^p\colon\sA_V^{p,0}(U)\to\sA_V^{0,p}(U)$ coincides with $p!\cdot\rJ\colon\sA_V^{p,0}(U)\to\sA_V^{0,p}(U)$.
\end{lem}

\begin{proof}
It is elementary.
\end{proof}

\begin{remark}
The presheaf $U\mapsto\sA_V^{p,q}(U)$ is already a sheaf on $V$. The maps $\rd',\rd'',\rJ,\wedge,\rN$ induce maps of sheaves with same relations. In particular, we have the map
\[\rN_V\colon\sA_V^{p,q}\to \sA_V^{p-1,q+1}\]
of sheaves for $p\geq 1$, and corresponding Lemma \ref{le:differential_local} and Lemma \ref{le:wedge_local}. Lemma \ref{le:functorial_local} induces the equality
\[(\varphi_*\rN_{V'})\varphi^*=\varphi^*\rN_V\colon\sA_V^{p,q}\to\varphi_*\sA_{V'}^{p-1,q+1}\]
on the level of sheaves.
\end{remark}


Let $P$ be a polyhedral complex (\emph{polytope} in \cite{CLD12}*{\Sec 1.1}) in $V$ and $j\colon P\to V$ the tautological inclusion map. For every open subset $U$ of $P$, let $\sN^{p,q}(U)$ be the subspace of $(j^{-1}\sA_V^{p,q})(U)$ of $(p,q)$-forms $\omega$ such that for every polyhedron $C$ of $P$, the restriction of $\omega$ to $\langle C\rangle$ is zero on $\langle C\rangle\cap U$ where $\langle C\rangle$ is the affine subset of $V$ spanned by $C$. It is clear that $\sN^{p,q}$ is a $\bR$-subsheaf of $j^{-1}\sA_V^{p,q}$. Let $\sA_P^{p,q}$ be the quotient sheaf $j^{-1}\sA_V^{p,q}/\sN^{p,q}$.

By Lemma \ref{le:functorial_local}, the monodromy map $j^{-1}\rN_V\colon j^{-1}\sA_V^{p,q}\to j^{-1}\sA_V^{p,q}$ preserves $\sN^{p,q}$. Therefore, we have an induced map
\[\rN_P\colon\sA_P^{p,q}\to\sA_P^{p-1,q+1}\]
of sheaves for $p\geq 1$.

\section{Monodromy maps for real forms on analytic spaces}
\label{ss:3}

In this section, we review the construction of $(p,q)$-forms in \cite{CLD12}, tropical Dolbeault cohomology, and introduce the monodromy map on analytic spaces.

We fix a non-Archimedean field $K$. Let $X$ be a $K$-analytic space. Recall that a \emph{tropical chart} of $X$ is given by a moment map $f\colon X\to T$ to a torus $T$ over $K$ and a compact polyhedral complex $P$ of $T_\trop$ that contains $f_\trop(X)$. Here $T_\trop$ is the tropicalization of $T$, which is a $\bR$-vector space of finite dimension, and $f_\trop\colon X\to T\to T_\trop$ is the composite map.

For every open subset $U$ of $X$, denote by $\sA_{\r{pre}}^{p,q}(U)$ the inductive limit of $\sA^{p,q}_P(P)$ for all tropical charts $(f\colon U\to T,P)$ of $U$. Again by Lemma \ref{le:functorial_local}, the monodromy maps $\rN_P$ are compatible with transition maps. Therefore, we obtain a map $\rN_{\r{pre}}\colon\sA_{\r{pre}}^{p,q}(U)\to\sA_{\r{pre}}^{p-1,q+1}(U)$ for $p\geq1$. The sheaf of $(p,q)$-forms on $X$ is defined as the sheafification of $U\mapsto\sA_{\r{pre}}^{p,q}(U)$, denoted by $\sA_X^{p,q}$.

\begin{definition}\label{de:monodromy}
For $p\geq 1$, we define the \emph{monodromy map} for forms on $X$, denoted by
\[\rN_X\colon\sA_X^{p,q}\to\sA_X^{p-1,q+1}\]
to be the sheafification of $U\mapsto[\rN_{\r{pre}}\colon\sA_{\r{pre}}^{p,q}(U)\to\sA_{\r{pre}}^{p-1,q+1}(U)]$.

For $0\leq r\leq p$, we denote by $\rN_X^r\colon\sA_X^{p,q}\to\sA_X^{p-r,q+r}$ the iterated composition.
\end{definition}

\begin{lem}\label{le:functorial}
We have
\begin{enumerate}
  \item The map $\rN_X^p\colon\sA_X^{p,q}\to\sA_X^{0,p}$ coincides with $p!\cdot\rJ$.

  \item $\rN_X$ commutes with $\rd''$, that is, $\rN_X\rd''=\rd''\rN_X\colon\sA_X^{p,q}\to\sA_X^{p-1,q+2}$ for $p\geq 1$.

  \item $(\rN_X-)\wedge -=-(-\wedge(\rN_X-))\colon\sA^{p,q}_X\times\sA^{p',q'}_X\to\sA_X^{p+p'-1,q+q'+1}$ for $p,p'\geq 1$ and $p+p'\geq\dim X+1$.

  \item Let $\varphi\colon X'\to X$ be a map of $K$-analytic spaces. Then
       \[(\varphi_*\rN_{X'})\varphi^*=\varphi^*\rN_X\colon\sA_X^{p,q}\to\varphi_*\sA_{X'}^{p-1,q+1}.\]
\end{enumerate}
\end{lem}

\begin{proof}
They are consequences of Lemma \ref{le:J_pre}, Lemma \ref{le:differential_local}, Lemma \ref{le:wedge_local} and Lemma \ref{le:functorial_local}, respectively.
\end{proof}

For a fixed integer $p$, we have the complex
\[
(\sA_X^{p,\bullet},\rd'')\colon \sA_X^{p,0}\xrightarrow{\rd''}\sA_X^{p,1}\xrightarrow{\rd''}\cdots.
\]

\begin{definition}[Dolbeault cohomology, \cite{Liu}]
Let $X$ be a $K$-analytic space. We define the \emph{Dolbeault cohomology}
(of forms) to be
\[H^{p,q}(X)\coloneqq\frac{\Ker(\rd''\colon\sA_X^{p,q}(X)\to\sA_X^{p+1}(X))}
{\IM(\rd''\colon\sA_X^{p,q-1}(X)\to\sA_X^{p,q}(X))}.\]
\end{definition}

The monodromy map in Definition \ref{de:monodromy} is in fact a map of complexes
\[
\rN_X\colon \sA_X^{p,\bullet}\to\sA_X^{p-1,\bullet}[1],
\]
and thus induces a map
\[\rN_X\colon H^{p,q}(X)\to H^{p-1,q+1}(X)\]
of Dolbeault cohomology when $p\geq 1$.

Let $\sO_X$ be the structure sheaf of $X$. For $p\geq 0$, let $\sO_X^{(p)}$ be the sheaf such that for every open subset $U$ of $X$, $\sO_X^{(p)}(U)$ is the $\bQ$-vector space spanned by symbols $\{f_1,\dots,f_p\}$ with $f_i\in\sO_X^*(U)$. For $p\geq 0$, we have a natural map
\[\tau\colon\sO_X^{(p)}\to\Ker[\rd''\colon\sA^{p,0}_X\to\sA^{p,1}_X]\]
of $\bQ$-sheaves on $X$. Let $\sT_X^p$ be its image sheaf. We recall the definition of $\tau$. For an open subset $U$ of $X$ and $f_1,\dots,f_p\in\sO^*_X(U)$, we have a moment map $f=(f_1,\dots,f_p)\colon U\to T=(\bG_m^\an)^p$. Let $\{x_1,\dots,x_p\}$ be the standard coordinates of $T_\trop=\bR^p$ \footnote{The map $\bG_m^\an\to(\bG_m^\an)_\trop=\bR$ is given by $-\log|\;|$.}. Then $\tau(\{f_1,\dots,f_p\})$ is defined as $\rd'x_1\wedge\cdots\wedge\rd'x_p$, regarded as an element in $\Ker[\rd''\colon\sA^{p,0}_X(U)\to\sA^{p,1}_X(U)]$.

\begin{proposition}
The canonical map
\[\sT_X^p\otimes_\bQ\bR\to\Ker[\rd''\colon\sA^{p,0}_X\to\sA^{p,1}_X]\]
is an isomorphism. It induces a canonical isomorphism
\[H^q(X,\sT_X^p)\otimes_\bQ\bR\simeq H^{p,q}(X).\]
\end{proposition}

\begin{proof}
By \cite{Jel16}*{Corollary 4.6} and \cite{CLD12}*{Corollaire 3.3.7}, the complex $(\sA_X^{p,\bullet},\rd'')$ is a fine resolution of
$\Ker[\rd''\colon\sA^{p,0}_X\to\sA^{p,1}_X]$. Thus the proposition follows from standard facts in homological algebra.
\end{proof}

\section{Semistable schemes and cohomological monodromy maps}
\label{ss:4}

In this section, we introduce some constructions for strictly semistable schemes, and review the (cohomological) monodromy maps coming from the weight spectral sequence.

Let $K$ be a discrete non-Archimedean field, that is, a non-Archimedean field with discrete valuation. Fix a rational prime $\ell$ that is invertible in $\widetilde{K}$. Denote by $K^\ur\subset K^\ra$ the maximal unramified extension with the residue field $\widetilde{K}^\rs$, which is a separable closure of $\widetilde{K}$. Let $\rI_K\subset\Gal(K^\ra/K)$ be the inertia subgroup, so the quotient group $\Gal(K^\ra/K)/\rI_K$ can be identified with $\Gal(\widetilde{K}^\rs/\widetilde{K})$. Denote by $t_\ell\colon\rI_K\to\bZ_\ell(1)$ the ($\ell$-adic) tame quotient homomorphism, that is, the one sending $\sigma\in\rI_K$ to $(\sigma(\varpi^{1/\ell^n})/\varpi^{1/\ell^n})_n$ for a uniformizer $\varpi$ of $K$. We fix an element $T\in\rI_K$ such that $t_\ell(T)$ is a topological generator of $\bZ_\ell(1)$.

For a separated scheme $X$ of finite type over $K$, put $X_\ra=X\otimes_KK^\ra$, and let $X^\an$ be the associated $K$-analytic space (\cite{Berk93}). Put $X^\an_\ra=(X\otimes_K\widehat{K^\ra})^\an$. For a scheme $\cX$ over $K^\circ$, put $\cX_\eta=\cX\otimes_{K^\circ}K$. For a scheme $Z$ over $\widetilde{K}$, put $Z_\rs=Z\otimes_{\widetilde{K}}\widetilde{K}^\rs$.

We first recall the following definition (see \cite{Sai03} for example).

\begin{definition}[Strictly semistable scheme]\label{de:semistable_schemes}
Let $\cX$ be a scheme locally of finite presentation over $\Spec K^\circ$. We say that $\cX$ is \emph{strictly semistable} if it is Zariski locally smooth over
\[\Spec K^\circ[t_0,\dots,t_p]/(t_0\cdots t_p-\varpi)\]
for some integer $p\geq 0$ (which may vary) and a uniformizer $\varpi$ of $K$.
\end{definition}

Let $\cX$ be a proper strictly semistable scheme over $K^\circ$. The special fiber $Y\coloneqq\cX\otimes_{K^\circ}\widetilde{K}$ is a normal crossing divisor of $\cX$. Suppose that $\{Y^1,\dots,Y^m\}$ is the set of irreducible components of $Y$. For a nonempty subset $I\subset\{1,\dots,m\}$, put $Y^I=\bigcap_{i\in I}Y^i$. For $p\geq 0$, put
\[
Y^{(p)}=\coprod_{I\subset\{1,\dots,m\},|I|=p+1}Y^I.
\]
Then $Y^{(p)}$ is a finite disjoint union of smooth proper subschemes of $Y$ of codimension $p$ over $\widetilde{K}$. Denote by $Y^{[p]}$ the image of the canonical morphism $Y^{(p)}\to Y$.

\begin{notation}\label{no:intersection}
For a subscheme $Z$ of $Y$, put $I(Z)=\{i\in\{1,\dots,m\}\res Z\subset Y^i\}$ and $p(Z)=|I(Z)|-1$.
\end{notation}

\begin{construction}\label{no:coordinate}
Let $y$ be a point of $Y$ with $p=p(y)\geq 0$. Suppose that $I(y)=\{i_0<\cdots<i_p\}$ (Notation \ref{no:intersection}). Choose an affine open neighborhood $\cV$ of $y$ in $\cX$ such that $\cV$ is smooth over $\Spec K^\circ[t_0,\dots,t_p]/(t_0\cdots t_p-\varpi)$ under which the divisor defined by $t_j$ is $\cV\cap Y^{i_j}$. Let $g_j$ be the restriction of $t_j$ to $\cV_\eta$.

We say that $(\cV,\{g_0,\dots,g_p\})$ is a \emph{semistable chart} at $y$ if moreover $\cV\cap W$ is either empty or connected for every irreducible component $W$ of $Y^{[r]}$ with $r\geq 0$.
\end{construction}

For subsets $J\subset I\subset\{1,\dots,m\}$ such that $|I|=|J|+1$, let $i_{JI}\colon\bigcap_{i\in I}Y^i\to\bigcap_{i\in J}Y^i$ denote the closed immersion. If $I=\{i_0<\cdots<i_p\}$ and $J=I\setminus\{i_j\}$, then we put $\epsilon(J,I)=(-1)^j$. We define the \emph{pullback map}
\begin{align}\label{eq:pullback}
\delta_p^*\colon H^q_{\et}(Y^{(p)}_\rs,\bQ_\ell)\to H^q_{\et}(Y^{(p+1)}_\rs,\bQ_\ell)
\end{align}
to be the alternating sum $\sum_{I\subset J,|I|=|J|-1=p+1}\epsilon(I,J)i_{IJ}^*$ of the restriction maps, and the \emph{pushforward map}
\begin{align}\label{eq:pushforward}
\delta_{p*}\colon H^q_{\et}(Y^{(p)}_\rs,\bQ_\ell)\to H^{q+2}_{\et}(Y^{(p-1)}_\rs,\bQ_\ell(1))
\end{align}
to be the alternating sum $\sum_{I\supset J,|I|=|J|+1=p+1}\epsilon(J,I)i_{JI*}$ of the Gysin maps. These maps satisfy the formula
\begin{align*}
\delta_{p-1}^*\circ\delta_{p*}+\delta_{p+1*}\circ\delta_p^*=0.
\end{align*}

Let us recall the following weight spectral sequence $E^{p,q}_\cX$ attached to $\cX$\footnote{It also depends on the ordering of the set of irreducible components of $Y$.}, originally studied in \cite{RZ82}:
\[
E^{p,q}_{\cX,1}=\bigoplus_{i\geq\max(0,-p)}H^{q-2i}_{\et}(Y_\rs^{(p+2i)},\bQ_\ell(-i))\Rightarrow H^{p+q}_{\et}(\cX_{\eta,\ra},\bQ_\ell).
\]
Here we will follow the convention and discussion in \cite{Sai03}. When $\cX$ is fixed, we write $E$ for $E_\cX$ for short. By \cite{Sai03}*{Corollary 2.8 (2)}, we have a map $\mu\colon E^{\bullet-1,\bullet+1}_\bullet\to E^{\bullet+1,\bullet-1}_\bullet$ of spectral sequences (depending on $T$). The differential map $d_1^{p,q}$ is an appropriate sum of pullback and pushforward maps. The map $\mu^{p,q}_1\colon\ E^{p-1,q+1}_1\to E^{p+1,q-1}_1$ is the sum of its restrictions to each direct summand $H^{q+1-2i}_{\et}(Y_\rs^{(2i+1)},\bQ_\ell(-i))$, and such restriction is the tensor product map by $t_\ell(T)$ (resp.\ the zero map) if $H^{q+1-2i}_{\et}(Y_\rs^{(2i+1)},\bQ_\ell(-i+1))$ does (resp.\ does not) appear in the target. For integers $p,q$ and $r\geq 0$, the map $\mu$ induces a map
\[\rN_\cX^r\colon E_2^{p,q}(r)\to E_2^{p+2r,q-2r}\]
which depends only on $\cX$.

\begin{conjecture}[Weight-monodromy conjecture]\label{co:purity}
Let $\cX$ be a proper strictly semistable scheme over $K^\circ$. Then for all integers $p,r\geq 0$, the map
\[\rN_\cX^r\colon E_2^{-r,p+r}(r)\to E_2^{r,p-r}\]
is an isomorphism.
\end{conjecture}

For $p\geq 1$, the map $\rN_\cX^p\colon E_2^{-p,2p}(p)\to E_2^{p,0}$ is simply the map
\begin{multline}\label{eq:weight}
E_2^{-p,2p}(p)=\Ker[H^0_{\et}(Y^{(p)}_\rs,\bQ_\ell)\xrightarrow{(\delta_p^*,\delta_{p*})} H^0_{\et}(Y^{(p+1)}_\rs,\bQ_\ell)\oplus H^2_{\et}(Y^{(p-1)}_\rs,\bQ_\ell(1))]\\
\to\frac{\Ker[H^0_{\et}(Y^{(p)}_\rs,\bQ_\ell)\xrightarrow{\delta_p^*}H^0_{\et}(Y^{(p+1)}_\rs,\bQ_\ell)]}
{\IM[H^0_{\et}(Y^{(p-1)}_\rs,\bQ_\ell)\xrightarrow{\delta_{p-1}^*}H^0_{\et}(Y^{(p)}_\rs,\bQ_\ell)]}=E_2^{p,0}
\end{multline}
induced by the identity map on $H^0_{\et}(Y^{(p)}_\rs,\bQ_\ell)$.

\begin{proposition}\label{pr:purity}
We have
\begin{enumerate}
  \item The spectral sequence $E_\cX$ degenerates from the second page. In particular, $E_2^{p,0}$ is canonically a subspace of $H^p_{\et}(\cX_{\eta,\ra},\bQ_\ell)$.

  \item If $K$ has equal characteristic, then Conjecture \ref{co:purity} holds.

  \item If $\widetilde{K}$ is a purely inseparable extension of a finitely generated extension of a prime field, then $\rN_\cX\colon E_2^{-1,2}(1)\to E_2^{1,0}$ is an isomorphism.
\end{enumerate}
\end{proposition}

\begin{proof}
See \cite{Ito05}*{Theorem 1.1} for (1) and (2). Part (3) follows from \cite{Ito05}*{Proposition 2.5 \& Remark 2.4}.
\end{proof}

Let $X$ be a separated scheme of finite type over $K$. From \cite{Berk00}*{\Sec 1}, we have the following map
\begin{align}\label{eq:comparison}
\kappa^p_X\colon H^p(X^\an,\bQ_\ell)\to H^p_{\et}(X_\ra,\bQ_\ell).
\end{align}
It is defined as the composition of the restriction map $H^p(X^\an,\bQ_\ell)\to H^p_{\et}(X^\an_\ra,\bQ_\ell)$ and the inverse of the comparison isomorphism $H^p_{\et}(X_\ra,\bQ_\ell)\xrightarrow{\sim} H^p_{\et}(X^\an_\ra,\bQ_\ell)$.

\begin{lem}\label{le:comparison}
The map $\kappa^p_{\cX_\eta}$ \eqref{eq:comparison} is injective with image contained in $E_2^{p,0}$. Moreover, if every irreducible component of $Y^{(r)}$ ($r\geq 0$) is geometrically irreducible, then $\kappa^p_{\cX_\eta}$ induces an isomorphism $H^p(\cX_\eta^\an,\bQ_\ell)\xrightarrow{\sim}E_2^{p,0}$.
\end{lem}

\begin{proof}
From \cite{Berk00}*{\Sec 4}, we have the following diagram
\[
\xymatrix{
H^p_{\r{Zar}}(Y,\bQ_\ell) \ar[r]^-{\simeq}\ar[d] & H^p(\cX_\eta^\an,\bQ_\ell) \ar[d] \\
H^p_{\r{Zar}}(Y_\rs,\bQ_\ell) \ar[r]^{\simeq}\ar[d] & H^p(\cX_{\eta,\ra}^\an,\bQ_\ell) \ar[d] \\
H^p_{\et}(Y_\rs,\bQ_\ell) \ar[r] & H^p_{\et}(\cX_{\eta,\ra},\bQ_\ell)
}\]
in which the composition of the two right vertical maps is just $\kappa^p_{\cX_\eta}$. By \cite{Berk00}*{Lemma 4.1}, the upper and middle horizontal maps are both isomorphisms. The map $H^p(\cX_\eta^\an,\bQ_\ell)\to H^p(\cX_{\eta,\ra}^\an,\bQ_\ell)$ is injective by the discussion after \cite{Berk00}*{Theorem 1.1}. By the proper descent and the fact that $H^i_{\r{Zar}}(Z,\bQ_\ell)=0$ for $i>0$ and $Z$ smooth over $\widetilde{K}^\rs$, the composite map
\[
H^p_{\r{Zar}}(Y_\rs,\bQ_\ell)\to H^p_{\et}(Y_\rs,\bQ_\ell)\to H^p_{\et}(\cX_{\eta,\ra},\bQ_\ell)
\]
is an isomorphism onto its image, which is $E_2^{p,0}$. Therefore, $\kappa^p_{\cX_\eta}$ \eqref{eq:comparison} is injective with image contained in $E_2^{p,0}$.

Moreover, if every irreducible component of $Y^{(r)}$ is geometrically irreducible, then the map
$H^p_{\r{Zar}}(Y,\bQ_\ell)\to H^p_{\r{Zar}}(Y_\rs,\bQ_\ell)$ is an isomorphism. The lemma follows.
\end{proof}

\begin{lem}\label{le:injective}
Suppose that one of following conditions holds:
\begin{enumerate}
  \item $K$ is a local non-Archimedean field;
  \item there is a finite extension $K'$ of $K$ such that $X\otimes_KK'$ is the generic fiber of a proper strictly semistable scheme over $K'^\circ$.
\end{enumerate}
Then the map $\kappa^p_X$ is injective for all $p\geq 0$.
\end{lem}

\begin{proof}
Case (1) follows from \cite{Berk00}*{Corollary 1.2}. Case (2) follows from Lemma \ref{le:comparison}, and the fact that the map $H^p(X^\an,\bQ_\ell)\to H^p((X\otimes_KK')^\an,\bQ_\ell)$ is injective.
\end{proof}

\section{Monodromy maps for tropical Dolbeault cohomology}
\label{ss:5}

In this section, we introduce the conjecture on the isomorphism of tropical Dolbeault cohomology groups under monodromy maps. Then we prove our main results.

Let $K$ be a non-Archimedean field. Let $X$ be a separated scheme of finite type over $K$.

\begin{definition}[Tropical Dolbeault cohomology]\label{de:dolbeault}
We define the \emph{tropical Dolbeault cohomology} of $X$ to be
\[H^{p,q}_\trop(X)\coloneqq H^{p,q}(X^\an),\]
and the corresponding \emph{tropical Hodge number} of $X$ to be
\[h^{p,q}_\trop(X)=\dim_\bR H^{p,q}_\trop(X).\]
We have the monodromy map
\[\rN_X=\rN_{X^\an}\colon H^{p,q}_\trop(X)\to H^{p-1,q+1}_\trop(X)\]
for $p\geq 1$.
\end{definition}

\begin{conjecture}\label{co:monodromy}
Suppose that $K$ is an algebraically closed non-Archimedean field such that $\widetilde{K}$ is algebraic over a finite field. Let $X$ be a proper smooth scheme over $K$. Then for $p\geq q\geq 0$, the (iterated) monodromy map
\[\rN_X^{p-q}\colon H^{p,q}_\trop(X)\to H^{q,p}_\trop(X)\]
is an isomorphism.
\end{conjecture}

\begin{remark}\label{re:counter}
Conjecture \ref{co:monodromy} does not hold for arbitrary algebraically closed non-Archimedean fields. In fact, let $X_0$ be the generic fiber of the scheme $\fX$ in \cite{BGS}*{(6.1)}, which is a geometrically connected projective smooth curve over $K_0=\bC(\!(t)\!)$. Take $K=\widehat{K_0^\ra}$ and $X=X_0\otimes_{K_0}K$. Then one can show that $h^{1,0}_\trop(X)=0$ but $h^{0,1}_\trop(X)=2$. This also implies that the canonical pairing $H^{1,0}_\trop(X)\times H^{0,1}_\trop(X)\to H^{1,1}_\trop(X)\xrightarrow{\Tr}\bR$ (where $\Tr$ is the integration map) is not a perfect pairing; in other words, the Poincar\'{e} duality fails.
\end{remark}

\begin{definition}\label{de:semistable_alteration}
Let $X$ be a proper smooth scheme over a discrete non-Archimedean field $K$. A \emph{strict semistable alteration} of $X$ is a proper strictly semistable scheme $\cX$ over $K_\cX^\circ$ for some finite extension $K_\cX/K$ contained in $K^\ra$ together with a proper dominant generically finite morphism $\phi\colon\cX_\eta\to X$ over $K$, such that every irreducible component of $Y^{(r)}$ for $r\geq 0$ is geometrically irreducible.
\end{definition}

\begin{theorem}\label{th:monodromy}
Let $X$ be a proper smooth scheme over a discrete non-Archimedean field $K$ and $p\geq 0$ an integer. Suppose that $\kappa^p_X$ \eqref{eq:comparison} is injective.
\begin{enumerate}
  \item If $\log|K^\times|\subset\bQ$, then the monodromy map $\rN_X^p\colon H^{p,0}_\trop(X)\to H^{0,p}_\trop(X)$ is rational, that is, it sends $H^0(X^\an,\sT_{X^\an}^p)$ into $H^p(X^\an,\bQ)$.

  \item Suppose that for every strict semistable alteration $\cX$ of $X$ (Definition \ref{de:semistable_alteration}), the map
      \[\rN_\cX^p\colon E_2^{-p,2p}(p)\to E_2^{p,0}\]
      is injective. Then $\rN_X^p\colon H^{p,0}_\trop(X)\to H^{0,p}_\trop(X)$ is injective.
\end{enumerate}
\end{theorem}

We need some preparation before the proof. The case $p=0$ is trivial. So we assume $p\geq 1$. The following notation will be used later.

\begin{notation}\label{no:tropicalization}
Let $h\colon T'\to T$ be a homomorphism of $K$-analytic torus. Then we denote by $h^\flat\colon T'_\trop\to T_\trop$ the induced linear map under tropicalization.
\end{notation}

Let $\cX$ be a proper strictly semistable scheme over $K_\cX^\circ$. We have the reduction map
\[\pi_\cX\colon\cX_\eta^\an\to Y.\]

Let $\omega$ be an element of $H^0(X^\an,\sT_{X^\an}^p)$. We say that a strict semistable alteration $\cX$ of $X$ (with a morphism $\phi\colon\cX_\eta\to X$) \emph{presents} $\omega$ if for every irreducible component $Y^i$ of $Y$, there exist
\begin{itemize}
  \item an open subset $U$ of $\cX_\eta^\an$ containing $\pi_\cX^{-1}Y^i$,

  \item $c_l\in\bQ$ for $1\leq l\leq M$ with some integer $M=M_U\geq 1$,

  \item $f_{lk}\in\sO^*_{\cX_\eta^\an}(U)$ for $1\leq l\leq M$ and $1\leq k\leq p$ satisfying $|f_{lk}|=1$ on $\pi_\cX^{-1}(Y^i\setminus Y^{[1]})$,
\end{itemize}
such that $\phi^{\an*}\omega\res_U=\tau\(\sum_{l=1}^Mc_l\{f_{l1},\dots,f_{lp}\}\)$. We call such data $(U,\{c_l\},\{f_{lk}\})$ a \emph{presentation} of $\omega$ on $Y^i$.

Now let $Z$ be an irreducible component of $Y^I$ for some $I=\{i_0<\cdots <i_p\}$ and $p\geq 1$. Choose a presentation $(U,\{c_l\},\{f_{lk}\})$ of $\omega$ on $Y^i$. For $j=1,\dots,p$, let $a_{lkj}\in\bZ$ be the order of zero (or the negative order of pole) of $f_{lk}$ along the connected component of $Y^{i_0}\cap Y^{i_j}$ containing $Z$.

\begin{lem}\label{le:order}
Let notation be as above. The rational number $\sum_{l=1}^Mc_l\det(a_{lkj})_{k,j=1}^p$ depends only on $\omega$ and $Z$.
\end{lem}

\begin{proof}
Let $\eta_Z$ be the generic point of $Z$. Then $I(\eta_Z)=\{i_0<\cdots <i_p\}$ (Notation \ref{no:intersection}). We choose a semistable chart $(\cV,\{g_0,\dots,g_p\})$ of $\eta_Z$ (Construction \ref{no:coordinate}). Put $V\coloneqq\cV\cap Y^{i_0}$, and replace $U$ by $U\cap\cV_\eta^\an$ which is an open neighborhood of $\pi_\cX^{-1}V$. We regard $g_1,\dots,g_p$ as elements in $\sO_{\cX_\eta^\an}^*(U)$. Then for $1\leq j\leq p$, $|g_j|=1$ on $\pi_\cX^{-1}(V\setminus Y^{[1]})$ and the divisor associated to the reduction $\widetilde{g_j}$ on $V$ is $V\cap Y^{i_j}$.

Note that the divisor associated to $\widetilde{f_{lk}}\res_V$ is supported on $V\cap(\bigcup_{j=1}^pY^{i_j})$. Put
\[f'_{lk}=f_{lk}\cdot\prod_{j=1}^pg_j^{-a_{lkj}}.\]
Then the reduction $\widetilde{f'_{lk}}$ is invertible on $V$. In other words, $|f'_{lk}|=1$ on $\pi_\cX^{-1}V$. Thus, $\tau(f'_{lk})=0$ on $\pi^{-1}_\cX V$. An elementary calculation shows that
\begin{align}\label{eq:order}
\phi^{\an*}\omega\res_{\pi^{-1}_\cX V}=
\tau\(\sum_{l=1}^Mc_l\{f_{l1},\dots,f_{lp}\}\)=\(\sum_{l=1}^Mc_l\det(a_{lkj})_{k,j=1}^p\)\tau(\{g_1,\dots,g_p\}).
\end{align}
It is clear that $\tau(\{g_1,\dots,g_p\})$ is nonzero as $V$ contains $\eta_Z$. The lemma then follows.
\end{proof}

We denote the rational number in the above lemma by $\ord_\omega(Z)$. The assignment $Z\mapsto\ord_\omega(Z)$ gives rise to an element $\ord_\omega\in H^0_{\et}(Y^{(p)}_\rs,\bQ_\ell)$. Denote by $H^0(X^\an,\sT_{X^\an}^p)_\cX$ the subset of $H^0(X^\an,\sT_{X^\an}^p)$ consisting of $\omega$ which $\cX$ presents. Then it is easy to see that $H^0(X^\an,\sT_{X^\an}^p)_\cX$ is a $\bQ$-subspace.

\begin{lem}\label{le:e2}
Let $\ord\colon H^0(X^\an,\sT_{X^\an}^p)_\cX\to H^0_{\et}(Y^{(p)}_\rs,\bQ_\ell)$ be the map sending $\omega$ to $\ord_\omega$.
\begin{enumerate}
  \item The map $\ord$ is injective.
  \item The image of $\ord$ is contained in the subspace $E_2^{-p,2p}(p)\subset H^0_{\et}(Y^{(p)}_\rs,\bQ_\ell)$ \eqref{eq:weight}.
\end{enumerate}
\end{lem}

\begin{proof}
For (1), take an element $\omega\in H^0(X^\an,\sT_{X^\an}^p)_\cX$. Let $x\in\cX_\eta^\an$ be a point, and put $y=\pi_\cX(x)\in Y$. Suppose that $I(y)=\{i_0<\cdots<i_r\}$ (Notation \ref{no:intersection}) for some $r\geq 0$, and choose a semistable chart $(\cV,\{g_0,\dots,g_r\})$ of $y$ (Construction \ref{no:coordinate}). Put $V\coloneqq\cV\cap Y^{i_0}$.

Choose a presentation $(U,\{c_l\},\{f_{lk}\})$ of $\omega$ on $Y^{i_0}$. We replace $U$ by $U\cap\cV_\eta^\an$, and view $g_1,\dots,g_r$ as in $\sO_{\cX_\eta^\an}^*(U)$. Then there exist unique integers $a_{lkj}$ such that if we put $f'_{lk}=\prod_{j=1}^rg_j^{a_{lkj}}$, then $|f'_{lk}|=|f_{lk}|$ on $\pi_\cX^{-1}V$. Let $f$ (resp.\ $f'$) be the moment map $\pi_\cX^{-1}V\to(\bG_m^\an)^{Mp}$ induced by $\{f_{lk}\}$ (resp.\ $\{f'_{lk}\}$). Then $f_\trop=f'_\trop$. In particular,
\[
\phi^{\an*}\omega\res_{\pi_\cX^{-1}V}=\tau\(\sum_{l=1}^Mc_l\{f'_{l1},\dots,f'_{lp}\}\).
\]
Let $g\colon\pi_\cX^{-1}V\to(\bG_m^\an)^r$ be the moment map induced by $\{g_1,\dots,g_r\}$. Then there exists a unique homomorphism $h\colon(\bG_m^\an)^r\to(\bG_m^\an)^{Mp}$ determined by integers $\{a_{lkj}\}$ such that $f_\trop=(h\circ g)_\trop$. Now if $\ord_{\omega}(Z)=0$ for every irreducible component $Z$ of $Y^{[p]}$, then the pullback of $\tau\(\sum_{l=1}^Mc_l\{f'_{l1},\dots,f'_{lp}\}\)$ under $h^\flat$ (Notation \ref{no:tropicalization}) is zero. In other words, we have $\phi^{\an*}\omega\res_{\pi_\cX^{-1}V}=0$. Thus the map $\ord\colon H^0(X^\an,\sT_{X^\an}^p)_\cX\to H^0_{\et}(Y^{(p)}_\rs,\bQ_\ell)$ is injective since $\pi_\cX^{-1}V$ is a neighborhood of $x$ and $x$ is arbitrary; (1) follows.

For (2), we need to show that $\delta_{p*}\ord_\omega=\delta_p^*\ord_\omega=0$. Suppose that $Z$ is an irreducible component of $Y^I$ for some $I=\{i_0<\cdots <i_p\}$. For every permutation $\sigma$ of the set $\{0,\dots,p\}$, we may define in the same way a rational number $\ord_\omega^\sigma(Z)$ by replacing $i_j$ by $i_{\sigma(j)}$. If $\sigma(0)=0$, then we let $\epsilon(\sigma)\in\{\pm1\}$ be the signature of the permiutation $\sigma\res_{\{1,\dots,p\}}$. If $\sigma(0)\neq 0$, then we let $\epsilon(\sigma)\in\{\pm1\}$ be the negative of the signature of the permutation from $\{1,\dots,0,\dots,p\}$ ($\sigma(0)$ is replaced by $0$) to $\{\sigma(1),\dots,\sigma(p)\}$. Then the proof of Lemma \ref{le:order} implies that
\begin{align}\label{eq:e2}
\ord_\omega^\sigma(Z)=\epsilon(\sigma)\cdot\ord_\omega(Z).
\end{align}

We start from $\delta_{p*}\ord_\omega$. Fix an irreducible component $W$ of $Y^J$ for some $J=\{i_0<\cdots<i_{p-1}\}$. For $1\leq j\leq p-1$, let $W_j$ be the unique irreducible component of $Y^{i_0}\cap Y^{i_j}$ that contains $W$. Choose a presentation $(U,\{c_l\},\{f_{lk}\})$ of $\omega$ on $Y^{i_0}$. By linear algebra, it is easy to see that we may choose $f_{lk}$ such that $\ord_{W_j}(\widetilde{f_{lk}})=0$ if $j<k$. In particular, the restriction of $\widetilde{f_{lp}}$ on $W$ is a nonzero rational function, which we denote by $f_{lp}^W$. Put $b_{l,j}=\ord_{W_j}(\widetilde{f_{lj}})$ for $1\leq j\leq p-1$. Let $Z$ be an irreducible component of $W\cap Y^{[p]}$. Then there is a unique element $i_p\in\{1,\dots,m\}\setminus\{i_0,\dots,i_{p-1}\}$ such that $Z\subset W\cap Y^{i_p}$. Let $0\leq \epsilon_Z\leq p$ be the integer such that there are exactly $\epsilon_Z$ elements in $\{i_0,\dots,i_{p-1}\}$ that are greater than $i_p$. Then by \eqref{eq:e2}, we have
\[
\ord_\omega(Z)=(-1)^{\epsilon_Z}\sum_{l=1}^M c_l b_{l,1}\cdots b_{l,p-1} \ord_Z(f_{lp}^W).
\]
Therefore, we have the equality
\[
(-1)^{p-\epsilon_Z}\sum_{Z\subset W}\ord_\omega(Z)Z=(-1)^p\sum_{l=1}^M c_l b_{l,1}\cdots b_{l,p-1}\DIV(f_{lp}^W)
\]
of divisors on $W$. Thus $\delta_{p*}\ord_\omega=0$ by \eqref{eq:pushforward}.

Now we consider $\delta_p^*\ord_\omega$. Let $W$ be an irreducible component of $Y^J$ for some $J=\{i_0<\cdots<i_{p+1}\}$. For $0\leq j\leq p+1$, let $Z_j$ be the unique irreducible component of $Y^{J\setminus\{i_j\}}$ containing $W$. Let $(\cV,\{g_0,\dots,g_{p+1}\})$ be a semistable chart of the generic point of $W$ (Construction \ref{no:coordinate}). Put $V=\cV\cap Y$ and $V_j=V\setminus Y^{i_j}$. From the proof for (1), we know that
\[\phi^{\an*}\omega\res_{\pi_\cX^{-1}V}=\sum_{\alpha}c_\alpha\tau(\{g_{\alpha(1)},\dots,g_{\alpha(p)}\})\]
for some $c_\alpha\in\bQ$, where the sum is taken over all strictly increasing maps $\alpha\colon\{1,\dots,p\}\to\{0,\dots,p+1\}$. We take such a map $\alpha$. Now let $\alpha_0<\alpha_1$ be the two integers in $\{0,\dots,p+1\}$ not in the image of $\alpha$. We have three cases:
\begin{itemize}
  \item If $j\in\IM(\alpha)$, then $\tau(\{g_{\alpha(1)},\dots,g_{\alpha(p)}\})\res_{\pi_\cX^{-1}V_j}=0$.

  \item If $j=\alpha_0$, then $\tau(\{g_{\alpha(1)},\dots,g_{\alpha(p)}\})\res_{\pi_\cX^{-1}V_j}=(-1)^{\alpha_1-1}\tau(\{g_{\alpha'(1)},\dots,g_{\alpha'(p)}\})$, where $\alpha'$ is the unique map whose image does not include $\alpha_0$ and $\min(\{0,\dots,p+1\}\setminus\{\alpha_0\})$.

  \item If $j=\alpha_1$, then $\tau(\{g_{\alpha(1)},\dots,g_{\alpha(p)}\})\res_{\pi_\cX^{-1}V_j}=(-1)^{\alpha_0}\tau(\{g_{\alpha'(1)},\dots,g_{\alpha'(p)}\})$, where $\alpha'$ is the unique map whose image does not include $\{0,\alpha_1\}$.
\end{itemize}
By \eqref{eq:order}, we have
\[
\ord_\omega(Z_j)=\sum_{\alpha,\alpha_0=j}c_\alpha(-1)^{\alpha_1-1}+\sum_{\alpha,\alpha_1=j}c_\alpha(-1)^{\alpha_0}.
\]
However, by \eqref{eq:pullback}, we have
\[
\delta_p^*\omega\res_W=\sum_{j=0}^{p+1}(-1)^j\ord_\omega(Z_j)
=\sum_\alpha c_\alpha(-1)^{\alpha_0+\alpha_1-1}+c_\alpha(-1)^{\alpha_0+\alpha_1}=0.
\]
Thus (2) follows.
\end{proof}

The map $\ord$ in Lemma \ref{le:e2} is inspired and closely related to the construction in \cite{Liu}*{\Sec 4}. The following lemma is the key step to the proof of Theorem \ref{th:monodromy}. It also justify the terminology \emph{monodromy map} for $\rN_X$ in Definition \ref{de:monodromy} and Definition \ref{de:dolbeault}.

\begin{lem}\label{le:rational}
Let $\cX$ be a strict semistable alteration of $X$, with $\phi\colon\cX_\eta\to X$. Suppose $\log|\varpi|=-1$ for a uniformizer $\varpi$ of $K$. Then the image of the restriction of $\rN_X^p\colon H^{p,0}_\trop(X)\to H^{0,p}_\trop(X)$ to $H^0(X^\an,\sT_{X^\an}^p)_\cX$ is contained in $H^p(X^\an,\bQ)$. Moreover, the following diagram
\begin{align}\label{eq:rational}
\xymatrix{
H^0(X^\an,\sT_{X^\an}^p)_\cX \ar[r]^-{\ord}\ar[d]_-{\rN_X^p} & E_2^{-p,2p}(p) \ar[d]^-{\rN_\cX^p} \\
H^p(X^\an,\bQ) \ar[r] & E_2^{p,0}
}
\end{align}
commutes. Here, the bottom map is the composition of
\begin{itemize}
  \item the multiplication map $(-1)^{p(p+1)/2}\colon H^p(X^\an,\bQ)\to H^p(X^\an,\bQ_\ell)$,

  \item the map $\kappa_X^p\colon H^p(X^\an,\bQ_\ell)\to H^p_{\et}(X_\ra,\bQ_\ell)$ \eqref{eq:comparison} (which is assumed to be injective), and

  \item the pullback map $\phi^*\colon H^p_{\et}(X_\ra,\bQ_\ell)\to H^p_{\et}(\cX_{\eta,\ra},\bQ_\ell)$.
\end{itemize}
The image of the bottom map is contained in $E_2^{p,0}$ by Lemma \ref{le:comparison}.
\end{lem}

\begin{proof}
To simplify notation, put $\sA_\bullet^{p,q,\cl}\coloneqq\Ker[\rd''\colon\sA_\bullet^{p,q}\to \sA_\bullet^{p,q+1}]$. Take an element $\omega\in H^0(X^\an,\sT_{X^\an}^p)_\cX$. By Lemma \ref{le:functorial} (1), $\rN_X^p(\omega)$ is represented by the Dolbeault representative $p!\rJ\omega\in\rH^0(X^\an,\sA_{X^\an}^{0,p,\cl})$.

Suppose that \eqref{eq:rational} commutes. By the projection formula, we have $\phi_*\circ\phi^*=[K_\cX(\cX_\eta):K(X)]\neq 0$. Thus $\phi^*$ is injective. Choose an embedding $\bR\hookrightarrow\bQ_\ell^\ra$. Then the map $H^p(X^\an,\bR)\to E_2^{p,0}\otimes_{\bQ_\ell}\bQ_\ell^\ra$ is injective. Thus $\rN_X^p(\omega)\in H^p(X^\an,\bQ)$.

Now we focus on the commutativity. For $1\leq i\leq m$, put $U^i=\pi_\cX^{-1}Y^i$, and choose a presentation $(U^i,\{c_l^i\},\{f_{lk}^i\}\res 1\leq l\leq M_i,1\leq k\leq p)$ of $\omega$ on $Y^i$. For every subset $I\subset\{1,\dots,m\}$, put $U^I=\bigcap_{i\in I}U^i=\pi_\cX^{-1}Y^I$.
Since $\pi_0(U^I)=\pi_0(Y^I)$, the assignment $Z\mapsto\ord_\omega(Z)$ gives rise to a \v{C}ech $p$-cocycle $\theta_\omega$ for the sheaf $\bQ$ with respect to the ordered open covering $\underline{U}=\{U^1,\dots,U^m\}$ of $\cX_\eta^\an$. Since $\delta_p^*\ord_\omega=0$ by Lemma \ref{le:e2} (2), $\theta_\omega$ is closed hence gives rise to a class in $H^p(\underline{U},\bQ)$, whose image $[\theta_\omega]\in H^p(\cX_\eta^\an,\bQ)$ coincides with $(\kappa^p_{\cX_\eta})^{-1}(\rN_\cX^p(\ord_\omega))$ where $\kappa^p_{\cX_\eta}\colon H^p(\cX_\eta^\an,\bQ_\ell)\xrightarrow{\sim} E_2^{p,0}$ is the isomorphism in Lemma \ref{le:comparison}. Therefore by Lemma \ref{le:functorial} (1), the commutativity of \eqref{eq:rational} is equivalent to that $(-1)^{p(p+1)/2}p!\rJ\omega$ is a Dolbeault representative of $[\theta_\omega]$.

Let us recall the construction of Dolbeault representatives. For an abelian sheaf $\sF$ on $\cX_\eta^\an$ and $r\geq 0$, denote by
\[C^r(\underline{U},\sF)=\bigoplus_{|I|=r+1}\Gamma(U^I,\sF)\]
the abelian group of \v{C}ech $r$-cocycles for $\sF$ with respect to $\underline{U}$. We have a coboundary map $\delta\colon C^r(\underline{U},\sF)\to C^{r+1}(\underline{U},\sF)$. By the Poincar\'{e} lemma (\cite{Jel16}*{Corollary 4.6}), we have a short exact sequence
\[
0 \to \sA^{0,r,\cl}_{\cX_\eta^\an} \to \sA^{0,r}_{\cX_\eta^\an} \xrightarrow{\rd''} \sA^{0,r+1,\cl}_{\cX_\eta^\an} \to 0.
\]
If we have elements $\theta_r\in C^r(\underline{U},\sA^{0,p-r-1}_{\cX_\eta^\an})$ for $0\leq r\leq p-1$ satisfying
\begin{align}\label{eq:rational1}
\rd''\theta_0=\rJ\omega,\quad \rd''\theta_1=\delta\theta_0,\quad\cdots\quad
\rd''\theta_{p-1}=\delta\theta_{p-2},\quad \frac{(-1)^{p(p+1)/2}}{p!}\theta_\omega=\delta\theta_{p-1},
\end{align}
then $(-1)^{p(p+1)/2}p!\rJ\omega$ is a Dolbeault representative of $[\theta_\omega]$ and \eqref{eq:rational} commutes. Here, we regard $\rJ\omega$ as an element in $C^0(\underline{U},\sA_{X^\an}^{0,p})$.

The remaining proof will be dedicated to the construction of $\theta_r$. For $r\geq 0$, denote by $\cD^r$ the set of irreducible components of $Y^{[r]}$.

\emph{Step 1.} For $Z\in\cD^r$, put $U_Z=\pi_\cX^{-1}Z$, $M_Z=\sum_{i\in I(Z)}M_i$ and let
\[
f_Z\colon U_Z\to(\bG_m^\an)^{M_Zp}
\]
be the moment map given by invertible functions $\{f^i_{lk}\res i\in I(Z),1\leq l\leq M_i,1\leq k\leq p\}$ on $U$. Put $U_Z^\circ=\pi_\cX^{-1}(Z\setminus Y^{[r+1]})$. The image of the tropicalization map $f_{Z,\trop}\colon U_Z\to\bR^{M_Zp}$, denoted by $C_Z$, is canonically a simplicial complex, induced by the reduced normal crossing divisor $Z\cup Y^{[r+1]}$ of $Z$, with the unique minimal simplex $\Delta_Z\coloneqq f_{Z,\trop}(U_Z^\circ)$. We now define the \emph{barycenter} $P_Z$ of $\Delta_Z$.

Let $\eta_Z$ be the generic point of $Z$. Let $(\cV,\{g_0,\dots,g_r\})$ be a semistable chart at $\eta_Z$. We let $g_Z\colon U_Z\cap\cV_\eta^\an\to(\bG_m^\an)^{r+1}$ be the moment map induced by $\{g_0,\dots,g_r\}$. Then the image of $g_{Z,\trop}$ is the standard (open) simplex $\Delta^\circ$ in $\bR^{r+1}$, which is $\{(x_0,\dots,x_r)\in\bR^{r+1}\res x_0+\cdots+x_r=1,x_i>0\}$. Similar to the proof of Lemma \ref{le:e2}, we have unique integers $a^i_{lkj}$ such that
\[f^i_{lk}\cdot\prod_{j=0,j\neq i}^r g_j^{-a^{i}_{lkj}}\]
has normal $1$ on $\pi_\cX^{-1}(Z\cap\cV)\subset U_Z^\circ$. Then there exists a unique homomorphism
\begin{align}\label{eq:rational4}
h_Z\colon(\bG_m^\an)^{r+1}\to(\bG_m^\an)^{M_Zp}
\end{align}
determined by integers $\{a^i_{lkj}\}$ such that $f_{Z,\trop}\res_{\pi_\cX^{-1}(Z\cap\cV)}=(h_Z\circ g_Z)_\trop$. Then we have \[h_Z^\flat(\Delta^\circ)=f_{Z,\trop}(\pi_\cX^{-1}(Z\cap\cV))\subset\Delta_Z.\]
We define $P_Z$ to be the image under $h_Z^\flat$ (Notation \ref{no:tropicalization}) of the barycenter of $\Delta^\circ$, which is $\{(\frac{1}{r+1},\dots,\frac{1}{r+1})\}$. It is easy to see that the point $P_Z$ does not depend on the choice of the semistable chart of $\eta_Z$ (and in fact, $h_Z^\flat(\Delta^\circ)=\Delta_Z$). To summarize, we have $P_Z\in\Delta_Z\subset C_Z\subset\bR^{M_Zp}$.

\emph{Step 2.} Take another irreducible component $Z'\in\cD^{r'}$ such that $Z\subset Z'$. In particular, we have $r'\leq r$ and $I(Z')\subset I(Z)$. Let
\begin{align}\label{eq:rational5}
h_{Z,Z'}\colon(\bG_m^\an)^{M_Zp}\to (\bG_m^\an)^{M_{Z'}p}
\end{align}
be the canonical projection by forgetting components with $i\in I(Z)\setminus I(Z')$. Then the following diagram
\[
\xymatrix{
U_Z \ar[r]^-{f_Z}\ar@{^(->}[d]_-{i_{Z,Z'}} & (\bG_m^\an)^{M_Zp} \ar[d]^-{h_{Z,Z'}} \\
U_{Z'} \ar[r]^-{f_{Z'}} & (\bG_m^\an)^{M_{Z'}p}
}
\]
commutes, where $i_{Z,Z'}$ is the natural inclusion map. It induces, for every $q\geq 0$, the following commutative diagram
\[
\xymatrix{
\sA^{0,q}_{\overline{C_{Z'}}}(C_{Z'}) \ar[r]^-{\iota_{Z'}}\ar[d]_-{(h_{Z,Z'}^\flat)^*} & \sA^{0,q}_{\cX_\eta^\an}(U_{Z'})  \ar[d]^-{i_{Z,Z'}^*} \\
\sA^{0,q}_{\overline{C_{Z}}}(C_{Z}) \ar[r]^-{\iota_Z} & \sA^{0,q}_{\cX_\eta^\an}(U_{Z})
}
\]
where $\iota_Z$ and $\iota_{Z'}$ are natural maps from the definition of $\sA^{0,q}_{\cX_\eta^\an}$.

On the other hand, we have a canonical isomorphism
\[
C^r(\underline{U},\sA^{0,q}_{\cX_\eta^\an})\simeq\bigoplus_{Z\in\cD^r}\sA^{0,q}_{\cX_\eta^\an}(U_Z),
\]
which induces a canonical map
\[
\iota^r\colon \bigoplus_{Z\in\cD^r}\sA^{0,q}_{\overline{C_Z}}(C_Z)\to C^r(\underline{U},\sA^{0,q}_{\cX_\eta^\an})
\]
for $r\geq 0$, and there is a similarly defined coboundary map
\begin{align}\label{eq:rational6}
\delta\colon\bigoplus_{Z\in\cD^r}\sA^{0,q}_{\overline{C_Z}}(C_Z)\to\bigoplus_{Z'\in\cD^{r+1}}\sA^{0,q}_{\overline{C_{Z'}}}(C_{Z'})
\end{align}
using $(h_{Z,Z'}^\flat)^*$ as restriction maps, such that it is compatible with the coboundary map for $C^\bullet(\underline{U},\sA^{0,q}_{\cX_\eta^\an})$ under $\iota^\bullet$.

Moreover, since $C_Z$ is star-shaped with respect to $P_Z$, we have the star-shape integration map
\[
\cI''_{P_Z}\colon\sA^{0,q}_{\overline{C_Z}}(C_Z)\to \sA^{0,q-1}_{\overline{C_Z}}(C_Z)
\]
for $q\geq1$ (see \Sec\ref{ss:6} for a formula on the standard simplex). We put
\[
\cI''\coloneqq\bigoplus\cI''_{P_Z}\colon\bigoplus_{Z\in\cD^r}\sA^{0,q}_{\overline{C_Z}}(C_Z)
\to\bigoplus_{Z\in\cD^r}\sA^{0,q-1}_{\overline{C_Z}}(C_Z).
\]

\emph{Step 3.} Note that $\cD^0=\{Y_i\res 1\leq i\leq m\}$. Define $\vartheta_0\in\bigoplus_{Z\in\cD^0}\sA^{0,p-1}_{\overline{C_Z}}(C_Z)$
by the formula
\[
\vartheta_0(Y^i)=\cI''_{P^{Y^i}}\(\sum_{l=1}^{M_i}c_l^i\;\rd''x_{l1}^i\wedge\cdots\wedge\rd''x_{lp}^i\),
\]
where $x^i_{lk}$ is the standard coordinates on $\bR^{M_ip}$. For $1\leq r\leq p-1$, we define
\[
\vartheta_r=\cI''(\delta\vartheta_{r-1})\in\bigoplus_{Z\in\cD^r}\sA^{0,p-r-1}_{\overline{C_Z}}(C_Z).
\]
We claim that
\begin{align}\label{eq:rational2}
\rd'(\delta\vartheta_{r-1})=\rd''(\delta\vartheta_{r-1})=0
\end{align}
for $1\leq r\leq p-1$, and
\begin{align}\label{eq:rational3}
\delta\vartheta_{p-1}=\frac{(-1)^{p(p+1)/2}}{p!}\theta_w.
\end{align}
We leave the proof to the next step. Assuming these, then $\rd''\vartheta_r=\delta\vartheta_{r-1}$ as $\rd''(\delta\vartheta_{r-1})=0$. Finally, we put $\theta_r=\iota^r\vartheta_r\in C^r(\underline{U},\sA^{0,p-r-1}_{\cX_\eta^\an})$ for $0\leq r\leq p-1$. Then they satisfies \eqref{eq:rational1}.

\emph{Step 4.} We have to verify \eqref{eq:rational2} and \eqref{eq:rational3} for $(\delta\vartheta_{r-1})(Z)\in\sA^{0,p-r}_{\overline{C_Z}}(C_Z)$ for every $Z\in\cD^r$. Without lost of generality, we assume that $X$ has dimension $n$. For each fixed $Z\in\cD^r$, it suffices to consider the restriction of $(\delta\vartheta_{r-1})(Z)$ to maximal (open) cells of $C_Z$, which are all of the form $f_{Z,\trop}(U_z)$ where $z\in Z$ is a closed point that belongs to $\cD^n$.

The ideal is to reverse the consideration. We fix a closed point $z\in\cD^n$ and consider the restriction of $(\delta\vartheta_{r-1})(Z)$ to $f_{Z,\trop}(U_z)$ for all $1\leq r\leq p$ and all $Z\in\cD^r$ such that $z\in Z$. We fix a semistable chart $(\cV,\{g_0,\dots,g_n\})$ at $z$. Then $\cV_\eta^\an$ contains $U_z$. We have the moment map $g\colon\cV_\eta^\an\to(\bG_m^\an)^{n+1}$ and $g_\trop\colon\cV_\eta^\an\to\bR^{n+1}$ so that $g_\trop(U_z)$ is the standard open simplex $\Delta^\circ$ in $\bR^{n+1}$. Note that for every subset $I\subset I(z)$ of cardinality $r+1$, there is a unique element $Z\in\cD^r$ that contains $z$, which we denote by $Z_I$. For every $I$ as above, put
\[
h_I\coloneqq h_{z,Z_I}\circ h_z\colon(\bG_m^\an)^{n+1}\to(\bG_m^\an)^{M_{Z_I}p}
\]
where $h_z$ and $h_{z,Z_I}$ are defined in \eqref{eq:rational4} and \eqref{eq:rational5}, respectively. Then $f_{Z,\trop}(U_z)$ is the image of $\Delta^\circ$ under $h_I^\flat$. Define the composite map
\[
h_I^\dag\colon\sA^{0,q}_{\overline{C_{Z_I}}}(C_{Z_I})\xrightarrow{(h_I^\flat)^*}\sA^{0,q}_\Delta(\Delta^\circ)\to\Omega^q(\Delta)
\]
where the second map is simply regarding a $(0,q)$-superform as an ordinary $q$-form (see \Sec\ref{ss:6} for notation concerning $\Delta$). Since $h_I^\flat$ is affine and $h_I(P^I)=P_{Z^I}$, the following diagram
\[
\xymatrix{
\sA^{0,q}_{\overline{C_{Z_I}}}(C_{Z_I}) \ar[r]^{h_I^\dag}\ar[d]_-{\cI''_{P_{Z_I}}} & \Omega^q(\Delta) \ar[d]^-{\cI_{P^I}} \\
\sA^{0,q-1}_{\overline{C_{Z_I}}}(C_{Z_I}) \ar[r]^{h_I^\dag} & \Omega^{q-1}(\Delta)
}
\]
commutes. Moreover, we have the following commutative diagram
\[
\xymatrix{
\bigoplus_{Z\in\cD^r}\sA^{0,q}_{\overline{C_Z}}(C_Z) \ar[r]^-{h_r}\ar[d]_-{\delta} & C^r(\Omega^q(\Delta)) \ar[d]^-{\delta} \\
\bigoplus_{Z'\in\cD^{r+1}}\sA^{0,q}_{\overline{C_{Z'}}}(C_{Z'}) \ar[r]^-{h_{r+1}} & C^{r+1}(\Omega^q(\Delta))
}
\]
concerning coboundary maps. Here, $\delta$ on the left-hand side is \eqref{eq:rational6}; $\delta$ on the right-hand side is defined in \Sec\ref{ss:6}; and $h_r$ is the composition of the projection
\[\bigoplus_{Z\in\cD^r}\sA^{0,q}_{\overline{C_Z}}(C_Z)\to\bigoplus_{I\subset I(z),|I|=r+1}\sA^{0,q}_{\overline{C_{Z_I}}}(C_{Z_I})\]
and $\bigoplus_{I\subset I(z),|I|=r+1}h_I^\dag$.

Take $\beta_0=h_0(\rd''\vartheta_0)$. Then $\beta_0$ comes from a unique element $\beta\in\Omega^p(\Delta)$ as in Proposition \ref{pr:star}. By the above discussion, we have
\[\beta_r=h_r(\delta\vartheta_{r-1})\]
for $1\leq r\leq p$, where $\beta_r$ is defined as in Proposition \ref{pr:star}. In particular, \eqref{eq:rational2} (resp.\ \eqref{eq:rational3}) holds when restricted to $f_{Z,\trop}(U_z)$ by Proposition \ref{pr:star} (1) (resp.\ (2)). Since $z$ is arbitrary, \eqref{eq:rational2} and \eqref{eq:rational3} hold, and the lemma follows.
\end{proof}

\begin{lem}\label{le:presentation}
We have $H^0(X^\an,\sT_{X^\an}^p)=\bigcup_\cX H^0(X^\an,\sT_{X^\an}^p)_\cX$, where the union is taken over all strict semistable alterations $\cX$ of $X$.
\end{lem}

\begin{proof}
Let $\omega$ be an element in $H^0(X^\an,\sT_{X^\an}^p)$. Since $X$ is proper, we may choose a finite open covering $\underline{U}$ of $X^\an$ such that for every $U\in\underline{U}$, there are $c_l\in\bQ$ for $1\leq l\leq M$ with some integer $M\geq 1$, and $f_{lk}\in\sO^*_{X^\an}(U)$ for $1\leq l\leq M$ and $1\leq k\leq p$, such that $\omega\res_U=\tau\(\sum_{l=1}^Mc_l\{f_{l1},\dots,f_{lp}\}\)$.

By \cite{Pay09}*{Theorem 4.2}, taking blow-ups, and possibly taking a finite extension $K'/K$ inside $K^\ra$, we have a (proper flat) integral model $\cX_0$ of $X\otimes_KK'$ such that if $Y_0^1,\dots,Y_0^{m_0}$ are all reduced irreducible components of $\cX_0\otimes_{K'^\circ}\widetilde{K'}$, then the covering $\{\pi_{\cX_0}^{-1}Y_0^i\res i=1,\dots,m_0\}$ refines $\underline{U}$. By \cite{dJ96}*{Theorem 8.2}, we have a proper strictly semistable scheme $\cX$ over $K_\cX^\circ$, where $K_\cX/K'$ is a finite extension inside $K^\ra$, with a proper dominant generically finite morphism $\cX\to\cX_0$. Replacing $K_\cX$ by a finite unramified extension inside $K^\ra$, we may assume that every irreducible component of $Y^{(p)}$ for $p\geq 0$ is geometrically irreducible. Therefore, $(\cX,\phi\colon\cX_\eta\to X)$ is a strict semistable alteration of $X$ such that $\{\pi_\cX^{-1}Y^i\res i=1,\dots,m\}$ refines $(\phi^\an)^{-1}\underline{U}$.

Now for an arbitrary irreducible component $Y^i$ of $Y$, take an open subset $U\in\underline{U}$ such that $\pi_\cX^{-1}Y^i\subset(\phi^\an)^{-1}U$. For every $f_{lk}$ as above, we may choose an element $a_{lk}\in K^*$ such that $|a_{lk}\phi^{\an*}f_{lk}|=1$ on $\pi_\cX^{-1}(Y^i\setminus Y^{[1]})$. Thus $((\phi^\an)^{-1}U,\{c_l\},\{a_{lk}\phi^{\an*}f_{lk}\})$ is a presentation of $\omega$ on $Y^i$. Therefore, $\omega$ belongs to $H^0(X^\an,\sT_{X^\an}^p)_\cX$.
\end{proof}

\begin{proof}[Proof of Theorem \ref{th:monodromy}]
Part (1) is a consequence of Lemmas \ref{le:rational} and \ref{le:presentation}. For (2), it is equivalent to check that the map $\rN_X^p\colon H^0(X^\an,\sT_{X^\an}^p)\to H^p(X^\an,\bQ)$ is injective. Let $\omega$ be an element of $H^0(X^\an,\sT_{X^\an}^p)$. By Lemma \ref{le:presentation}, we may assume that $\omega\in H^0(X^\an,\sT_{X^\an}^p)_\cX$ for a strict semistable alterations $\cX$ of $X$. Then it follows from Lemmas \ref{le:e2} and \ref{le:rational} as we assume that $\rN_\cX^p$ is injective.
\end{proof}

\begin{corollary}\label{co:injective}
Let $X_0$ be a proper smooth scheme over $K_0$ that is isomorphic to $k(\!(t)\!)$ for $k$ either a finite field or a field of characteristic zero. Let $K$ be a closed subfield of $\widehat{K_0^\ra}$ containing $K_0$. Then the monodromy map
\[\rN_X^p\colon H^{p,0}_\trop(X)\to H^{0,p}_\trop(X)\]
is injective, where $X=X_0\otimes_{K_0}K$. In particular, we have $h^{p,0}_\trop(X)<\infty$ and $h^{p-q,q}_\trop(X)\geq h^{p,0}_\trop(X)$ for $0\leq q\leq p$.
\end{corollary}

\begin{proof}
Suppose first $K/K_0$ is a finite extension. By Lemma \ref{le:injective} and resolution of singularity when $\Char k=0$, the map $\kappa_X^p$ is injective. Then the corollary follows from Theorem \ref{th:monodromy} and Proposition \ref{pr:purity} (1).

In general, since $X_0$ is proper, we have that
\[H^{p,0}_\trop(X)=\bigcup_{K_0\subset K'\subset K}H^{p,0}_\trop(X_0\otimes_{K_0}K')\]
where the union is taken over all finite extensions $K'/K_0$ contained in $K$. By \cite{Berk99}*{Theorem 10.1}, we know that $H^{p,0}_\trop(X_0\otimes_{K_0}K')\simeq H^p(X_0^\an\otimes_{K_0}K',\bR)$ stabilizes when $K'$ increases. Therefore, the injectivity of $\rN_X^p$ is reduced to the case where $K/K_0$ is finite.
\end{proof}

In the case where $p=1$, we have the following stronger result over certain non-Archimedean fields.

\begin{theorem}\label{th:isomorphic}
Let $K_0$ be a local non-Archimedean field, and $X_0$ a proper smooth scheme over $K_0$ that admits a proper strictly semistable model over $K_0^\circ$. Then the monodromy map
\[\rN_X\colon H^{1,0}_\trop(X)\to H^{0,1}_\trop(X)\]
is an isomorphism, where $X=X_0\otimes_{K_0}\widehat{K_0^\ra}$.
\end{theorem}

\begin{proof}
Since $X_0$ admits a proper strictly semistable model over $K_0^\circ$, it is well-known (see for example \cite{Sai03}*{Lemma 1.11}) that $X_0\otimes_{K_0}K$ admits a proper strictly semistable model over $K^\circ$ for any finite extension of $K$ of $K_0$. Replacing $K$ by a finite unramified extension, we may assume that $X_0$ admits a strict semistable alteration $\cX$ over $K^\circ$ such that $\phi\colon\cX_\eta\to X_0$ induces an isomorphism $\cX_\eta\xrightarrow{\sim}X_0\otimes_{K_0}K$. We then have an infinite sequence of successive finite field extensions $K_0\subset K_1\subset\cdots$ contained in $K_0^\ra$ such that for $i\geq 1$, $X_0$ admits a strict semistable alteration $\cX_i$ over $K_i^\circ$ with $\phi\colon\cX_{i,\eta}\xrightarrow{\sim} X_i$ being an isomorphism where $X_i\coloneqq X_0\otimes_{K_0}K_i$.

By the similar argument in the proof of Corollary \ref{co:injective}, it suffices to show that
\begin{align}\label{eq:isomorphism1}
\rN_{X_i}\colon H^{1,0}_\trop(X_i)\to H^{0,1}_\trop(X_i)
\end{align}
is an isomorphism for every $i\geq 1$. By Theorem \ref{th:monodromy} (2), Lemma \ref{le:injective} and Proposition \ref{pr:purity} (3), we know that \eqref{eq:isomorphism1} is injective. Thus if $h^{1,0}_\trop(X_i)\geq h^{0,1}_\trop(X_i)$, then \eqref{eq:isomorphism1} is an isomorphism.

We fix such an index $i$ and suppress from notation. In particular, now we have $X=X_0\otimes_{K_0}K$ with $K$ a finite extension of $K_0$. Let $\cX$ be a strict semistable alteration of $X$ with an isomorphism $\phi\colon\cX_\eta\xrightarrow{\sim}X$. We identify $X$ with $\cX_\eta$. By Lemma \ref{le:comparison} and Proposition \ref{pr:purity} (3), the composite map
\[
\rN_\cX^{-1}\circ\kappa_X^1\colon H^1(X^\an,\bQ_\ell)\to E_2^{-1,2}(1)
\]
is an isomorphism, under which the image of $H^1(X^\an,\bQ)$ is contained in $E_2^{-1,2}(1)\cap H^0_{\et}(Y^{(1)}_\rs,\bQ)$, which we denote by $E_\cX$. The theorem will follow if we can construct an injective (linear) map
\[
\rT_\cX\colon E_\cX\to H^0(X^\an,\sL^1_{X^\an}).
\]

For every point $y\in Y$, we fix a semistable chart $(\cV_y,\{g^y_0,\dots,g^y_{p(y)}\})$ at $y$ (Construction \ref{no:coordinate}). We can write an element in $E_\cX$ in the form
\[
D=\frac{1}{q}\sum_{Z\in\cD^1}a(Z)Z
\]
for some integers $a(Z)$, $q>0$, and the sum is taken over $\cD^1$, the set of all irreducible components of $Y^{[1]}$. As $\delta_{1*}D=0$ \eqref{eq:weight}, the divisor $(qD)\cap Y^i$ on $Y^i$ is cohomologically trivial for every irreducible component $Y^i$ of $Y$ ($1\leq i\leq m$). Then there exists some integer $q_i>0$ such that $(q_iqD)\cap Y^i$ is algebraically equivalent to zero; and in particular $\sO_{Y^i}((q_iqD)\cap Y^i)$ is an element in $\Pic^0_{Y^i/\widetilde{K}}(\widetilde{K})$. Since $\Pic^0_{Y^i/\widetilde{K}}$ is a projective scheme over the \emph{finite} field
$\widetilde{K}$, one may replace $q_i$ by some integer multiple such that $\sO_{Y^i}((q_iqD)\cap Y^i)$ is a trivial line bundle. Replacing $q$ by $q\prod_{i=1}^mq_i$, we may assume that for every $i$, there is a rational function $f_i\in\widetilde{K}(Y^i)$ such that its divisor is exactly $(qD)\cap Y^i$.

For every point $y$ of $Y$, we are going to construct an open neighborhood $U_y\subset X^\an$ of $\pi_\cX^{-1}y$ and an element $\omega_y\in\sL_{X^\an}^1(U_y)$. For every nonempty subset $I$ of $I(y)$ (Notation \ref{no:intersection}), let $Z_y^I$ be the unique irreducible component of $Y^I$ that contains $y$. Put $U_y\coloneqq\cV_{y,\eta}^\an\cap\pi_\cX^{-1}Z_y^{I(y)}$. For $\omega_y$, there are two cases:
 \begin{enumerate}
  \item If $p(y)=0$, that is, $Z_y=Y^i$ for some $i$, then we set
      \[
      \omega_y=\tau\(\frac{1}{q}\{f_i^\sharp\}\)\res_{U_y}
      \]
      where $f_i^\sharp$ is a lift of $f_i\res_{\cV_y\cap Y^i}$ to an invertible regular function on $\cV_{y,\eta}$.

  \item If $p(y)>0$, then we set
      \[
      \omega_y=\tau\(\frac{1}{q}\prod_{j=1}^{p(y)}g_j^{a(Z_y^{\{i_0,i_j\}})}\)\res_{U_y}
      \]
      where $I(y)=\{i_0<\cdots<i_{p(y)}\}$. Note that $Z_y^{\{i_0,i_j\}}\in\cD^1$.
\end{enumerate}
We claim that $\{(U_y,\omega_y)\}$ patch together to give an element in $H^0(X^\an,\sL^1_{X^\an})$, which we denote by $\rT_\cX D$. Given two points $y,y'$ of $Y$, it amounts to showing that $\omega_y=\omega_{y'}$ on $U_y\cap U_{y'}$. If one of $y,y'$ is in Case (1), then the compatibility can be shown in the same way as in the proof of Lemma \ref{le:order}. In particular, $\omega_y$ does not depend on the choice of the lift $f_i^\sharp$ in Case (1). If both $y$ and $y'$ are in Case (2), then it is a straightforward consequence of the fact that $\delta_1^*D=0$ \eqref{eq:weight}. We leave those details to the reader.

By construction, $\rT_\cX$ is apparently linear and injective. The theorem is proved.
\end{proof}

\section{Appendix: Star-shaped integration on simplex}
\label{ss:6}

In this appendix, we prove a technical result used in the proof of Lemma \ref{le:rational} as the local computation.

Let $n\geq 1$ be an integer. Let $\Delta\subset\bR^{n+1}$ be the standard (closed) simplex of dimension $n$. If we denote by $x_0,\dots,x_n$ the standard coordinates on $\bR^{n+1}$, then
\[\Delta=\{(x_0,\dots,x_n)\res x_0+\cdots+x_n=1,x_i\geq 0\}.\]
For $p\geq 0$, denote by $\Omega^p(\Delta)$ the space of (smooth) $p$-differential forms on $\Delta$ and $\Omega^p_\cl(\Delta)$ the subspace of closed forms. Let $\omega$ be an element of $\Omega^p(\Delta)$, then for every $0\leq i\leq n$, $\omega$ can be uniquely written as
\[
\omega=\sum_{I\subset\{0,\dots,n\}\setminus\{i\},|I|=p}f_I(x)\rd x_I
\]
where $f_I$ is a continuous function on $\Delta$, smooth in the interior of $\Delta$. We say that $\omega$ has \emph{constant coefficients} if $f_I$ is a constant function for every such $I$. Note that this property is independent of the choice of $i$. Let $\Omega^p_\cons(\Delta)\subset\Omega^p_\cl(\Delta)$ be the subspace of forms with constant coefficients.

For every nonempty subset $J\subset\{0,\dots,n\}$, let $\Delta^J$ be the face of $\Delta$ generated by $\{P^j\res j\in J\}$, where $P^j$ is the point with coordinates $x_i=\delta_{ij}$, and let $P^J$ be the barycenter of $\Delta^J$. For every point $P\in\Delta$, $\Delta$ is star-shaped with respect to $P$. Thus, one has the \emph{star-shaped integration map}:
\[
\cI_P\colon \Omega^p(\Delta)\to \Omega^{p-1}(\Delta)
\]
for $p\geq 1$. We briefly recall the definition. Consider the map $\gamma_P\colon \Delta\times[0,1]\to\Delta$ sending $(x,t)$ to $(1-t)P+tx$. Then
\[\cI_P(\alpha)=\int_0^1\langle\frac{\partial}{\partial t},\gamma_P^*\alpha\rangle\;\rd t,\]
where $\langle\;,\;\rangle$ denotes the contraction. If $\rd\alpha=0$, then $\rd(\cI_P(\alpha))=\alpha$.

Let $F$ be an abelian group. For every $r\geq 0$, put
\[C^r(F)=\Map(\{J\subset\{0,\dots,n\}\res|J|=r+1\},F).\]
We have a coboundary map $\delta\colon C^r(F)\to C^{r+1}(F)$ similar to the one for \v{C}ech cocycles. For $p-r\geq 1$, we define the map
\begin{align}\label{eq:star0}
\cI\colon C^r(\Omega^{p-r}(\Delta))\to C^r(\Omega^{p-r-1}(\Delta))
\end{align}
sending $\omega_r$ in the source to $\cI\omega_r$ in the target such that $(\cI\omega_r)(J)=\cI_{P^J}(\omega_r(J))$.

\begin{proposition}\label{pr:star}
Let $1\leq p\leq n$ be an integer. Let $\beta$ be an element in $\Omega^p_\cons(\Delta)$. We define $\beta_0\in C^0(\Omega^p(\Delta))$ by $\beta_0(\{j\})=\beta$ for every $0\leq j\leq n$. For $1\leq r\leq p$, we inductively define $\beta_r\in C^r(\Omega^{p-r}(\Delta))$ by the formula $\beta_r=\delta(\cI\beta_{r-1})$. Then
\begin{enumerate}
  \item $\beta_r$ belongs to $C^r(\Omega^{p-r}_\cons(\Delta))$ for $0\leq r\leq p$, and in particular $\beta_p\in C^p(\bR)$; and

  \item we have the formula
      \[\beta\res_{\Delta^I}=(-1)^{p(p+1)/2}p!\cdot\beta_p(I)\;\rd x_{i_1}\wedge\cdots\wedge\rd x_{i_p}\]
      for every subset $I=\{i_0<\dots<i_p\}\subset\{1,\dots,n\}$.
\end{enumerate}
\end{proposition}

\begin{proof}
We fix some $i\in\{0,\dots,n\}$ and write
\begin{align}\label{eq:star7}
\beta=\sum_{I\subset\{0,\dots,n\}\setminus\{i\},|I|=p}f_I\rd x_I
\end{align}
for some $f_I\in\bR$. Regard $\beta$ as a $p$-form on $\bR^{n+1}$ with constant coefficients, which we denote by $\beta'\in \Omega^p_\cons(\bR^{n+1})$. For a point $Q\in\bR^{n+1}$ and $r\geq 1$, denote by
\[\cI'_Q\colon \Omega^r(\bR^{n+1})\to \Omega^{r-1}(\bR^{n+1})\]
the corresponding star-shaped integral on $\bR^{n+1}$ with respect to $Q$. If $Q$ belongs to $\Delta$, then for $\alpha\in\Omega^r(\bR^{n+1})$, we have
\begin{align}\label{eq:star1}
\cI'_Q(\alpha)\res_{\Delta}=\cI_Q(\alpha\res_{\Delta}).
\end{align}

Suppose that $Q=(q_0,\dots,q_n)$; we define a map
\[\cC_Q\colon \Omega^r_\cons(\bR^{n+1})\to \Omega^{r-1}_\cons(\bR^{n+1})\]
sending $\alpha$ to the contraction of $\alpha$ by $\sum_{j=0}^nq_j\frac{\partial}{\partial x_j}$ (from left). For example,
\[
\cC_{P^j}(\rd x_0\wedge\cdots\wedge\rd x_{r-1})=
\begin{cases}
(-1)^j\rd x_0\wedge\cdots\wedge\widehat{\rd x_j}\wedge\cdots\wedge\rd x_{r-1},\qquad &0\leq j\leq r-1;\\
0, &r\leq j\leq n.
\end{cases}
\]
Denote by $O$ the origin of $\bR^{n+1}$. It is an elementary exercise to see that for $\alpha\in\Omega^r_\cons(\bR^{n+1})$, we have
\begin{align}\label{eq:star2}
\cI'_Q(\alpha)-\cI'_O(\alpha)=-\frac{1}{r}\cC_Q(\alpha).
\end{align}
For $0\leq r\leq p$, we define an element $\beta'_r\in C^r(\Omega^{p-r}_\cons(\bR^{n+1}))$ by the formula
\begin{align}\label{eq:star6}
\beta'_r(I)=\frac{(-1)^r}{p(p-1)\cdots(p-r+1)}\sum_{j=0}^r(-1)^j
(\cC_{P^{i_0}}\circ\cdots\circ\widehat{\cC_{P^{i_j}}}\circ\cdots\circ\cC_{P^{i_r}})(\beta').
\end{align}
for $I=\{i_0<\cdots<i_r\}$. We claim that
\begin{align}\label{eq:star3}
\beta'_r=\delta(\cI'\beta'_{r-1})
\end{align}
where $\cI'\colon C^r(\Omega^{p-r}(\bR^{n+1}))\to C^r(\Omega^{p-r-1}(\bR^{n+1}))$ is defined similarly as \eqref{eq:star0}. In fact, by the definition of $\delta$, we have for $r\geq 1$,
\begin{align}\label{eq:star4}
\delta(\cI'\beta'_{r-1})&=\sum_{j=0}^r(-1)^j\cI'_{P^{I\setminus\{i_j\}}}(\beta'_{r-1}(I\setminus\{i_j\})).
\end{align}
As $\beta'_{r-1}$ is a closed cocycle, we have $\sum_{j=0}^r(-1)^j\beta'_{r-1}(I\setminus\{i_j\})=0$. Thus
\begin{align}\label{eq:star5}
\eqref{eq:star4}&=\sum_{j=0}^r(-1)^j(\cI'_{P^{I\setminus\{i_j\}}}-\cI'_O)(\beta'_{r-1}(I\setminus\{i_j\})) \notag\\
&=\frac{-1}{p-r+1}\sum_{j=0}^r(-1)^j\cC_{P^{I\setminus\{i_j\}}}(\beta'_{r-1}(I\setminus\{i_j\}))\qquad \text{by \eqref{eq:star2}} \notag\\
&=\frac{(-1)^r}{p(p-1)\cdots(p-r+1)}\sum_{j=0}^r(-1)^j\cC_{P^{I\setminus\{i_j\}}}
\(\sum_{k=0}^r\epsilon(j,k)\cC(j,k)(\beta')\).
\end{align}
Here, $\epsilon(j,k)$ is set to be $(-1)^k$ (resp.\ $0$, $(-1)^{k+1}$) if $k<j$ (resp.\ $k=j$, $k>j$), and $\cC(j,k)$ is the map
$\cC_{P^{i_0}}\circ\cdots\circ\cC_{P^{i_r}}$ with $\cC_{P^{i_j}}$ and $\cC_{P^{i_k}}$ removed. Since
\[\cC_{P^{I\setminus\{i_j\}}}=\frac{1}{r}(\cC_{P^{i_0}}+\cdots+\widehat{\cC_{P^{i_j}}}+\cdots+\cC_{P^{i_r}}),\]
it is easy to see that
\[
\eqref{eq:star5}=\frac{(-1)^r}{p(p-1)\cdots(p-r+1)}\sum_{j=0}^r(-1)^j\cC(j,j)(\beta)=\beta'_r(I).
\]
Thus \eqref{eq:star3} holds. By \eqref{eq:star1}, we have $\beta_r=\beta'_r\res_\Delta$ which belongs to $C^r(\Omega^{p-r}_\cons(\Delta))$. Thus (1) follows.

Now we compute the value of $\beta_p(I)=\beta'_p(I)$ for $I=\{i_0<\cdots<i_p\}$. Note that the representative $\beta$ \eqref{eq:star7} can be chosen with respect to an arbitrary $i\in\{0,\dots,n\}$. In particular, we may take $i=i_0$. By \eqref{eq:star6}, we have
\[\beta_p(I)=\frac{(-1)^p}{p!}(\cC_{P^{i_1}}\circ\cdots\circ\cC_{P^{i_p}})(\beta)
=\frac{(-1)^{p(p+1)/2}}{p!}(\cC_{P^{i_p}}\circ\cdots\circ\cC_{P^{i_1}})(\beta).\]
Thus (2) follows as
\[\beta\res_{\Delta^I}=(\cC_{P^{i_p}}\circ\cdots\circ\cC_{P^{i_1}})(\beta)\;\rd x_{i_1}\wedge\cdots\wedge\rd x_{i_p}.\]
\end{proof}

\end{document}